\documentclass{elsarticle}
\usepackage{amsmath,amssymb,amsthm,mathrsfs}
\usepackage{xcolor}
\usepackage{tikz,url}
\usepackage{float}
\usepackage[latin1]{inputenc}
\usetikzlibrary{arrows,shapes,chains}
\theoremstyle{plain}
\newtheorem{theorem}{Theorem}[section]
\newtheorem{lemma}[theorem]{Lemma}
\newtheorem{proposition}[theorem]{Proposition}
\newtheorem{corollary}[theorem]{Corollary}

\theoremstyle{definition}
\newtheorem{definition}[theorem]{Definition}

\newtheorem{example}[theorem]{Example}

\newtheorem{claim}{Claim}[section]
\newtheorem{question}[theorem]{Question}

\newcommand{\ua}{\mathord{\uparrow}}
\newcommand{\da}{\mathord{\downarrow}}
\newcommand{\n}{\mathbb{N}}
\newcommand{\kf}{{ \rm KF}}
\newcommand{\KF}{{\rm KF}_c}
\newcommand{\WD}{{\rm WD}_c}
\newcommand{\wds}{{\rm WD}}
\newcommand{\Max}{\mathrm{Max}}
\newcommand{\Irr}{{\rm Irr}_c}
\newcommand{\irr}{{\rm Irr}}
\newcommand{\R}{\mathbb{R}}
\newcommand{\cl}{{\rm cl}}
\newcommand{\op}{\mathcal{O}}
\newcommand{\si}{\mathcal{S}_c}
\newcommand{\K}{\mathcal{K}}
\newcommand{\X}{Zh(X)}
\newcommand{\Z}{\widehat{Zh(X)}}
\newcommand{\G}{\mathcal{G}}
\newcommand{\pp}{\hat{P}}
\newcommand{\A}{\mathcal{A}}
\newcommand{\B}{\mathcal{B}}
\newcommand{\Y}{\mathbb{Y}}
\newcommand{\kk}{\mathbb{K}}
\newcommand{\Aa}{\mathbb{A}}
\newcommand{\Bb}{\mathbb{B}}
\makeatletter
\def\ps@pprintTitle{%
	\let\@oddhead\@empty
	\let\@evenhead\@empty
	\def\@oddfoot{\reset@font\hfil\thepage\hfil}
	\let\@evenfoot\@oddfoot
}
\makeatother

\begin{document}
	\begin{frontmatter}
		\title{Xi-Zhao Model Preserves WD Spaces\tnoteref{t1}}%The Xi-Zhao Models of ${\rm H}$-model Spaces and Weak ${\rm H}$-model Spaces
		\tnotetext[t1]{This work is supported by National Natural Science Foundation of China (No.12231007) }
		\author{Siheng Chen}
		\ead{mathlife@sina.cn}
\address{School of Mathematics and Statistics, Changsha University of Science and Technology, Changsha, Hunan, 410114, China}
		\author{Qingguo Li\corref{a1}}
		\address{School of Mathematics, Hunan University, Changsha, Hunan, 410082, China}
		\cortext[a1]{Corresponding author.}
		\ead{liqingguoli@aliyun.com}
		\begin{abstract}
For a $T_1$ space $X$, Zhao and Xi constructed a dcpo model $\pp$, where $P$ is a  bounded complete algebraic poset model of $X$.
In this paper, we  formulate  the closed WD subsets of  the maximal point space $\Max(\pp)$ and the Scott space $\Sigma \pp$,  and then prove that $X$ is a WD space if and only if $\Sigma \pp$ is  a WD space.
 It is also shown that the sobrification $X^s$ (resp., the well-filtered reflection $X^w$)  of $X$ can be embedded into the sobrification $\pp^s$  (resp., the well-filtered reflection $\pp^w$) of $\pp$ as a saturated subspace.
Finally, we  introduce two new concepts ``${\rm H}$-model spaces" and ``weak ${\rm H}$-model spaces", and provide a general framework to prove that such $T_1$ spaces can be preserved by Xi-Zhao dcpo models.
%			For a $T_1$ space $X$, Zhao and Xi constructed a dcpo model $\pp$ where $P$ is its bounded complete algebraic poset model.
%			Based on the Xi-Zhao model $\pp$, we obtain the following results:
%			(1) the sobrification $X^s$ can be embedded into $\pp^s$ as a saturated subspace;
%			(2) the well-filterification $X^w$ can be embedded into $\pp^w$ as a saturated subspace;
%			(3) $X$ is a WD space (resp., WK space) if and only if $\Sigma \pp$ is  a WD space (resp., WK space).
%			In addition, we give a general framework for prove some properties that Xi-Zhao model preserves.
			% some equational characterizations of the closed KF-subsets and the closed WD subsets of $\pp$ and $\Max(\pp)$.
		\end{abstract}
	\end{frontmatter}

	\section{Introduction}
The study of the relationship between  topological spaces and  order structures is a hot topic in domain theory.
The maximal point space which was founded in (\cite{weihrauch_embedding_1981}) is an important bridge from topological structure to ordered structure.
%It provides the characterization of those topological spaces  which are homeomorphic to the subspace of maximal elements from certain topology, such as the Scott topology.
A \emph{poset model} of a topological space $X$ is a poset $P$ such that the set $\Max(P)$ of all maximal points of $P$ equipped with the relative Scott topology is homeomorphic to $X$.

For a certain space , there always exists a pertinency tool to build its poset model.
%For different  $T_1$  spaces, there are own pertinency tools to build poset models.
%	In the early days of domain theory Dana Scott suggested the possibility of using domains to study computability on metric spaces(\cite{scott_outline_1970}). One can embed a topological space to a domain that allows one to model computational algorithms and study computational questions. Therefore, the study of the relationship between  topological spaces and  order structures is a hot topic in domain theory.
%	The maximal point space founded in (\cite{weihrauch_embedding_1981}) is one of the important branches for research.
%	It provides the characterization of those topological spaces  which are homeomorphic to the subspace of maximal
%	elements from certain topology, such as the Scott topology.
%	A \emph{poset model} of a topological space $X$ is a poset $P$ such that the set $\Max(P)$ of all maximal points of $P$ equipped with the relative Scott topology is homeomorphic to $X$.
%	In addition, several people gave other definitions  in different topologies, such as the lower topology(see \cite{li20}), the  strong Scott topology(see \cite{zhao_topologies_2020}).
%	
%	However, the work about the Scott topology is developed further.
Edalat and  Heckmann \cite{Edalat98} proved that for every metric space $X$, the formal ball $\mathbf{B}X$ is a continuous poset model of $X$.
Liang and Keimel \cite{Liang04} obtained a continuous poset model of the Tychonoff space by constructing a special abstract basis.
%	Martin \cite{martin03} showed that any $T_1$ space embeddable in a domain as a $G_\delta$ subset has an ideal model.
%	Several authors showed that $T_1$ topological spaces have a poset model (see \cite{zhao09},\cite{ali-akbari_any_2009},\cite{erne_algebraic_2011}).
%		
For each $T_1$ space $X$, Zhao \cite{zhao09} constructed a  bounded complete algebraic poset model $\X$ composed by all nonempty intersection filters of the open set lattice $\mathcal{O}(X)$.
% and showed that $\X$ is a poset model of  $X$.
%	

	For a given poset model of  a $T_1$  space $X$, a natural  question to ask is: if  $X$ has topological property $\mathcal{P}$, does the poset model have $\mathcal{P}$?
Positive answer to this problem for "first countable`` have been given in \cite{zhao09}.
In addition, the work about the Xi-Zhao dcpo model is very fruitful.
Based on \cite{zhao09}, Zhao and Xi proposed a new dcpo model  and revealed that a $T_1$ space $X$ is sober if and only if it has a sober dcpo model\cite{zhao18}.
%Successively, serval properties give positive answers, such as well-filtereness, K-closedness, Baire property, Choquet-completeness, weak sobriety and RD space(see \cite{xi17},\cite{he19},\cite{chen22}).%which were explored in the future work,
The other positive answers are successively explored, such as weak well-filtereness, Baire property, Choquet-completeness and RD spaces(see \cite{xi17,he19,chen22,chen_xi-zhao_2023}).

In this paper, we consider the previous question related to  ``WD spaces" for the Xi-Zhao dcpo model.
%However,  it is difficult to study  properties directly through the definition of  WD spaces.
It is  known that WD sets are deeply linked to the well-filtered reflections. %the well-filtered reflections of $T_0$ spaces can be constructed by  closed WD subsets.
%However, \cite{} presents an indirect construction, which is equivalent to that of \cite{xu_t0_2020}.%, according to the uniqueness (up to isomorphism) of the well-filtered reflection.
Motivated by \cite{xu_t0_2020, shen_well-filtered_2019},  we combine two constructions of the well-filtered reflections of a $T_1$ space and its Xi-Zhao dcpo model to propose an equation characterization of closed WD subsets.
Then we prove that a $T_1$ space $X$ is a WD space if and only if the Xi-Zhao dcpo model of $X$ is  a WD space.
Moreover, we introduce two new kinds of $T_0$ spaces ---- ``${\rm H}$-model spaces" and ``weak ${\rm H}$-model spaces".
According to ${\rm H}$-model spaces, we provide a uniform approach to sober spaces, well-filtered spaces, Rudin spaces and WD spaces.
As to weak ${\rm H}$-model spaces, weak sober spaces and weak well-filtered spaces are included.
Finally, we show that a $T_1$ space $X$ is a ${\rm H}$-model space (resp., a weak ${\rm H}$-model space) if and only if the Xi-Zhao dcpo model of $X$ is  a ${\rm H}$-model space (resp., a weak ${\rm H}$-model space).
%and provide a general framework to prove  that these spaces can be preserved by their Xi-Zhao dcpo  models.
	\section{Preliminaries}
	Let's recall  some basic concepts and notations that will be used in this paper. For further details, we refer the reader to \cite{gierz_continuous_2003,goubault-larrecq_non-hausdorff_2013}
	
	A subset $D$ of a poset $P$ is \emph{directed} (resp., \emph{filtered}) if every nonempty finite subset of $D$ has an upper bound (resp., a lower bound) in  $D$.
	A poset $P$ is called a \emph{directed complete poset}, or a \emph{dcpo} for short, if  every directed subset has a supremum.
	For any subset $A$ of a poset $P$, we set
	$\ua A= \{ x\in P : x\geqslant a ~\mathrm{for ~some}~ a \in A\}$, and dually, $\da A= \{ x\in P : x\leqslant a ~\mathrm{for ~some}~ a \in A\}$.
	We write $\ua\{x\}=\ua x$ and $\da\{x\}=\da x$ for short.
	A subset $A$ is called a \emph{ lower set} (resp. an \emph{upper set}) if $A=\da A$ (resp. $A=\ua A$).
	A poset $P$ is called \emph{bounded complete}, if every subset with an upper bound has a supremum.
	
	A subset $U$ of a poset $P$ is \emph{Scott open} if  (i) $U $ is an upper set and  (ii)  for any directed subset $D$, $\sup D \in U $ implies $D \cap U \neq \emptyset$, whenever $\sup D$ exists.
	All Scott open subsets form a topology on $P$, denoted by $\sigma (P)$ and called the \emph{Scott topology} on $P$.
	The space $(P, \sigma (P))$ is written as $\Sigma P$.
	
	%For a poset $P$ and $a$, $b\in P$,  we say that $a$ is \emph{way below} $b$ (in symbols $a \ll b$) if for any directed subset $D$ with $\sup D$ exists, $b \leqslant \sup D$  implies $\ua a \cap D \neq \emptyset$.
	%An element $a$ is called \emph{compact} if $a \ll a$.
	%The set of all compact elements of $P$ is denoted by $K(P)$.
	%A poset $P$ is called \emph{continuous} if for any $x \in P$, the set $\dda x=\{y \in P: y\ll x\}$ is directed and $x=\sup \dda x$.
	%A poset $P$ is called \emph{algebraic} if for any $x \in P$, the set $\da x\cap K(P)$ is directed and $x=\sup (\da x\cap K(P))$.
	
	%For a dcpo $P$ and $x\in P$, $F \subseteq P$, we say that $F$ is \emph{way below} $x$ ( in symbols $F \ll x$) if for any directed subset $D$, $x \leqslant \sup D$  implies $\ua F \cap D \neq \emptyset$.
	%A dcpo $P$ is called \emph{quasicontinuous} if for any $x \in P$, the family
	%$$\mathrm{fin}(x)=\{F \subseteq P: F\ll x ~\rm{and ~\emph{F} ~ is ~ finite}\}$$
	% is directed,  in the sense that for any $F_1$, $F_2 \in \mathrm{fin}(x)$, there exists $F_3 \in \mathrm{fin}(x)$ such that $F_3 \subseteq \ua F_1\cap\ua F_2$, and $\ua x=\bigcap \{\ua F \in P: F \in \mathrm{fin}(x)\}$.
	%A poset $P$ is called \emph{locally quasicontinuous} if for any $x \in P$, the poset $\da x$ is quasicontinuous.
	
	For a topological  space $X$, a nonempty subset $A$ is \emph{irreducible}  if $A \subseteq B\cup C$ implies $A \subseteq B$ or $A \subseteq C$, for any closed subsets $B$ and $C$.
	We denote the set of all irreducible  sets (resp., irreducible closed sets) of $X$ by $\irr(X)$ (resp., $\Irr(X)$).
	% Let $Q(X)$  (resp., $Q^*(X)$) be the set of all compact saturated  (resp., nonempty compact saturated) subsets of $X$ and $\op(X)$ the set of all open subsets  of $X$.
	Let $Q(X)$ be the set of all compact saturated subsets of $X$ and $\op(X)$ the set of all open subsets  of $X$.
	%Let $\op(X)$ (resp., $\Gamma(X)$) be the set of all open subsets (resp., closed subsets) of $X$.	

	\begin{lemma}[\cite{gierz_continuous_2003}]\label{irrset}
		For a topological space $X$, let $A\subseteq Y\subseteq X$. Then the following statements are equivalent:
		\begin{enumerate}[(1)]
			\item $A$ is irreducible in $Y$;
			\item $A$ is irreducible in $X$;
			\item $\cl_X(A)$ is irreducible in $X$.
		\end{enumerate}
	\end{lemma}
	
	\begin{lemma}(\cite{erne_sober_nodate,heckmann_quasicontinuous_2013})\label{rudinlemma}
		Let $X$ be a topological space and $\K$ a filtered family of nonempty compact subsets of $X$. Any closed set $C$ that meets all members of $\K$ contains an 	irreducible closed subset $A$ that still meets all members of $\K$.
	\end{lemma}	
	
	\begin{definition}(\cite{shen_well-filtered_2019})
		A nonempty subset $A$ of a $T_0$  space $X$ is said to be a \emph{KF-set} if there exists a filtered family $\K\subseteq Q(X)$ such that  $\overline{A}$ is a minimal closed set that intersects all members of $\K$ ($\overline{A}\in m(\K)$ for short).
	\end{definition}
	Clearly, $A$ is a KF-set iff $\overline{A}$ a KF-set. Denote by $\kf(X)$ (resp., $\KF(X)$) the set of all  KF-subsets (resp., closed KF-subsets) of $X$.
	\begin{lemma}(\cite{shen_well-filtered_2019,xu_t0_2020})\label{map}
		Let $X$, $Y$ be two $T_0$ spaces and $f: X\longrightarrow Y$ a continuous mapping. The following statements  hold:
		\begin{enumerate}[(1)]
			\item if $A\in \irr(X)$, then $f(A)\in \irr(Y)$;
			\item if $A\in \kf(X)$, then $f(A)\in \kf(Y)$;
		\end{enumerate}
	\end{lemma}
	For a topological space $X$, define $\mathcal{S}(X)=\{\{x\}:x\in X\}$ and $\si(X)=\{\overline{\{x\}}:x\in X\}$.
	For $\G\subseteq 2^{X}$ and $A\subseteq X$, let $\Box_{\G}A=\{G\in\G:G\subseteq A\}$ and $\Diamond_{\G}A=\{G\in\G: A\cap G\neq\emptyset\}$.
	The \emph{lower Vietoris topology} on $\G$ is the topology that has $\{\Diamond_{\G}U: U\in \op(X)\}$ as a subbase, and the resulting space is denoted by $P_H(\G)$.
	\begin{lemma}\label{embedding}(\cite{goubault-larrecq_non-hausdorff_2013})
		Let $X$ be a $T_0$ space.
		\begin{enumerate}[(1)]
			\item If $\G\subseteq \Irr(X)$, then $\{\Diamond_{\G}U: U\in \op(X)\}$ form a topology on $\G$ and $\{\Box_{\G}A:A{\rm ~is~ closed~ in~ }X\}$ are exactly all closed subsets of $\G$.
			\item  If $\si(X)\subseteq \G$, then the specialization order on $P_H(\G)$ is the set inclusion order, and the \emph{canonical
				mapping} $\eta_{X} : X \longrightarrow P_H(\G)$, given by $\eta_{X}(x) = \overline{\{x\}}$, is an order and topological embedding.
		\end{enumerate}
	\end{lemma}

	%We denote $\mathcal{O}(X)$ (resp., $\mathcal{O}^*(X)$)  the collection of all open (resp., nonempty open) subsets
	%and $Q(X)$ (resp., $Q^*(X)$)  the collection of all compact saturated (resp., nonempty compact saturated) subsets.
	%%For $A\subseteq X$, let $\N (A)=\{U \in \mathcal{O}(X) : A\subseteq U\}$.
	%%For $x \in X$, we write $\N (x)$ for $\N (\{x\})$.
	\begin{lemma}\label{closure}
		For a topological space $X$, let $Y$ be a subspace of $X$ and $A$ closed in $Y$. Then $A=\cl_X(A) \cap Y$.
	\end{lemma}
	\begin{lemma}\label{inject}
		Let $X$, $Y$ be two topological spaces and $f:X\longrightarrow Y$ an injective map. Then
		\begin{enumerate}[(1)]
			\item $f(\bigcap_{i\in I}A_i)=\bigcap_{i\in I} f(A_i)$ and $f(\bigcup_{i\in I}A_i)=\bigcup_{i\in I} f(A_i)$ for $\{A_i\}_{i\in I}\subseteq 2^X$;
			\item $f^{-1}(f(A))=A$ for $A\subseteq X$.
		\end{enumerate}
	\end{lemma}
	
	Given a poset $P$,  denote by $\Max (P)$ the set of all maximal points of $P$ equipped with the relative Scott topology from $P$.
	We write $Q(P)$ (resp., $\Irr(P)$, $\WD(P)$, $\KF(P)$) as short for $Q(\Sigma P)$  (resp., $\Irr(\Sigma P)$, $\WD(\Sigma P)$, $\KF(\Sigma P)$).
	
	\begin{definition}
		A poset $P$ is called a \emph{poset model} of a topological space $X$ if  $\Max(P)$ is homeomorphic to $X$.
	\end{definition}
	
	\begin{lemma}(\cite{zhao09})\label{posetmodel}
		For a $T_1$ space $X$, let $\X=\{ \mathcal{F}: \mathcal{F} $ is  a  filter  in $ \op(X)$ and $ \bigcap\mathcal{F}\neq\emptyset\}$.
		Then $(\X, \subseteq)$ is a bounded complete algebraic poset model of $X$.
	\end{lemma}
	
	\begin{lemma}(\cite{zhao18})\label{dcpomodel}%(\cite{xi17, zhao18})
		For a bounded complete algebraic poset $P$, define $\hat{P}$ as
		$$\hat{P}=\{(x,e) : x \in P, e \in \Max (P)~ \mathrm{and}~ x \leqslant_{P} e\}.$$
		And $(x,e) \leqslant (y,d)$ in $\hat{P}$ iff either $e=d$ and $x\leqslant_{P} y$, or $y=d$ and $x \leqslant_{P} d$.
		Then $\hat{P}$ is a dcpo and  the following properties hold:
		\begin{enumerate}[(1)]
			\item $\Max (\hat{P}) =\{(e,e) : e \in \Max (P) \}$ is homeomorphic to $\Max (P)$;
			\item for a directed subset $D$ of $\hat{P}$, either $D$ has a maximal element,
			or $D= \{ (x_i,e): i \in I\}$ for some $e \in \Max (P)$ and $\{x_i : i \in I\}$ is directed in $P$;
			%\item if $k\in K(P)$ and $d\in \ua k \cap \Max (P)$, then $(k, d)\ll (k, d)$ in $\da (d,d)$;
			% \item for each $(x,e)\in\hat{P}$, $\Sigma \da (x,e)$ is sober and the Scott topology of $\da (x,e)$ equals to the relative Scott  topology from $\hat{P}$;
			% \item $\Irr (\hat{P}) = \{\cl_{\hat{P}}(A):A\in \Irr (\Max (\hat{P}))\}\cup\{\da (x,e): (x,e)\in \hat{P}\}$,
			%\item $\Irr (\Max (\hat{P}))=  \{A\cap\Max (\hat{P}): A\cap\Max (\hat{P})\neq\emptyset ~\rm {and}~ A\in \Irr (\hat{P}) \}.$
		\end{enumerate}
	\end{lemma}

	\begin{lemma}(\cite{zhao18})\label{dcpomodel2}
		For any $T_1$ space $X$, there is a bounded complete algebraic poset $P$ such that $X$ is homeomorphic to $\Max(\pp)$.
		Thus $\pp$ is a dcpo model of $X$.
	\end{lemma}
	
	We call $\pp$ the \emph{Xi-Zhao model} of a $T_1$ space $X$ if $P$ is a bounded complete algebraic poset model of $X$. In particular, $\Z$ is a special Xi-Zhao  model of $X$.
	
	Let $P$ be a  poset and $A$ an upper set.
	Denote $$E_A=\{(e,e)\in \Max(\pp):{\rm~there~ exists ~a~point~}(x,e) {\rm~in~} A\backslash \Max(\pp)\}.$$
	Obviously, $E_A \subseteq  \Max(\pp)$, and if $A_1$, $A_2$ are two upper subsets of $P$, then $E_{A_1}\subseteq E_{ A_2}$ holds for $A_1\subseteq A_2$.
	\begin{lemma}(\cite{xi17})\label{compactset}
		Let $P$ be a bounded complete algebraic poset and  $K\in Q(\pp)$. Then   $E_K$ is finite.
	\end{lemma}
	
	%\begin{lemma}(\cite{zhao18})\label{sobermodel}
	%A $T_1$ space $X$ is sober if and only if $X$ has a sober dcpo model.
	%\end{lemma}
	%
	%\begin{lemma}(\cite{xi17})
	%A $T_1$ space $X$ is well-filtered if and only if $X$ has a well-filtered dcpo model.
	%\end{lemma}

	%\section{Sobrifications}
	\section{WD Spaces}
In this section, we consider the  ``WD spaces", lying between $T_0$ spaces and sober spaces and being different from well-filtered
spaces. We show that a $T_1$ space $X$  is a WD space if and only if its Xi-Zhao  model  is a WD space.

First, we introduce the concept and  some basic properties of WD spaces.

		A topological space $X$ is called \emph{sober} if for any  irreducible closed subset $A$, there exists a unique point $x \in X$ such that $A=\overline{\{x\}}$, that is, $\Irr(X)=\si(X)$. 	A topological space $X$ is called \emph{well-filtered} if for any filtered family $\K\subseteq Q(X)$ and any open set $U$, $\bigcap\K\subseteq U$ implies $K\subseteq U$ for some $K\in \K$.

	A  KF-set is also called a \emph{Rudin set} in \cite{xu_first-countability_2021}, in order to emphasize its origin from topological variant of Rudin's Lemma (i.e., Lemma \ref{rudinlemma}).	
	\begin{definition}(\cite{xu_t0_2020,xu_first-countability_2021})
		A $T_0$ space $X$ is called a \emph{Rudin space} if every irreducible closed set is a Rudin set, that is, $\KF(X)=\Irr(X)$.
	\end{definition}
	\begin{definition}(\cite{xu_t0_2020})
		A nonempty subset $A$ of a $T_0$  space $X$ is said to be a \emph{WD set} if for any continuous mapping $f: X\longrightarrow Y$ to a well-filtered space $Y$, there exists a unique $y_A\in Y$ such that $\overline{f(A)}=\overline{\{y_A\}}$.
	\end{definition}
	Similarly, $A$ is a WD set iff $\overline{A}$ is a WD set. Denote by $\wds(X)$ (resp., $\WD(X)$) the set of all  WD subsets (resp., closed WD subsets) of $X$.
	\begin{definition}(\cite{xu_t0_2020})
		A $T_0$ space $X$ is called a \emph{WD space} if every irreducible closed set is a WD set, that is, $\WD(X)=\Irr(X)$.
	\end{definition}	
	\begin{lemma}(\cite{xu_t0_2020})\label{map2}
		Let $X$, $Y$ be two $T_0$ spaces and $f: X\longrightarrow Y$ a continuous mapping.  If $A\in \wds(X)$, then $f(A)\in \wds(Y)$.
	\end{lemma}
	\begin{lemma}(\cite{xu_t0_2020})\label{subset}
		Let $X$ be a $T_0$ space. Then $\si(X)\subseteq \KF(X)\subseteq \WD(X)\subseteq\Irr(X)$.
	\end{lemma}
It follows that every Rudin space is a  WD space.
\begin{lemma}(\cite{shen_well-filtered_2019,xu_t0_2020})\label{wf}
		For a $T_0$ space $X$, the following conditions are equivalent:
		\begin{enumerate}[(1)]
			\item $X$ is well-filtered;
			\item $\KF(X)=\si(X)$;
			\item $\WD(X)=\si(X)$.
		\end{enumerate}
	\end{lemma}

%In \cite{chen22}, we have shown that a $T_1$ space $X$  is a Rudin space if and only if its Xi-Zhao dcpo model  is a Rudin space.
From \cite{chen22},   irreducible closed subsets  can be formulated as follows.

\begin{lemma}\label{irr}(\cite{chen22})
		Let $P$ be a bounded complete algebraic poset,  $A\in \Irr(\pp)$ and
		$A^*=A\cap \Max (\pp)$.
		The following statements hold:
		\begin{enumerate}[(1)]
			\item if $A^*=\emptyset$, then there exists $(x,e) \in A\backslash \Max (\pp)$ such that $A=\da (x,e)$;
			\item if $A^*\neq\emptyset$, then $A=\cl_{\hat{P}}(A^*)$ and $A^*$ is irreducible closed in $\Max (\pp)$.
		\end{enumerate}
	\end{lemma}
	By the above result and Lemma \ref{closure}, we can conclude the following equations:
	\begin{corollary}\label{equation1}(\cite{chen22})
		Let $P$ be a bounded complete algebraic poset. Then
		\begin{equation*}\left\{
			\begin{aligned}
				\Irr (\pp) =& \{\cl_{\pp}(A):A\in \Irr (\Max (\pp))\}\cup\{\da (x,e): (x,e)\in \pp\backslash \Max (\pp)\}, \\
				\Irr (\Max (\pp))= & \{A\cap\Max (\pp): A\cap\Max (\pp)\neq\emptyset ~\rm {and}~ A\in \Irr (\pp) \}.
			\end{aligned}
			\right.\end{equation*}
	\end{corollary}
Naturally, we have the following question.

Are there equation characterizations similar to above for the closed KF-subsets or the closed WD subsets of a $T_1$ space and its Xi-Zhao model?

Next, we will investigate some properties of the irreducible closed subsets and closed KF-subsets of a special $T_0$ space $\Y$ so as to solve the above question, which alos can  be used in the proof of our main result.

Let $P$ be a bounded complete algebraic poset.  A pair $\langle \Y, f\rangle $ consisting of a $T_0$ space $\Y$ and a  mapping $f: \Sigma\pp\longrightarrow \Y$ satisfies the following three conditions:
	\begin{itemize}
		\item[(P1)] $f:\Sigma\pp\longrightarrow \Y$ is a topological embedding ;
		\item[(P2)] $\Y=\ua_{\Y}f(\Max(\pp))\cup f(\pp\backslash \Max (\pp))$ and $f(\pp)$ is a lower subset of $\Y$;
		\item[(P3)] if $C$ is Scott closed in $\pp\backslash \Max (\pp)$, then $f(C)$ is closed in $\Y$.
	\end{itemize}

	\begin{lemma}
		Let  $\Y$ be a $T_0$ space and $P$  a bounded complete algebraic poset. For a pair $\langle \Y, f\rangle $ satisfying (P1)$\sim$(P3), if $A$ is a lower subset of $\pp$, then $f(A)$ is a lower subset of $\Y$.
	\end{lemma}
	\begin{proof}
Assume $(x,e)\in A$, let $y\in \Y$ with $y\leqslant_{\Y}f((x,e))$. Then (P2) implies $y\in f(\pp)=\da_{\Y} f(\pp)$,  which means there exists a point $(x',e')\in\pp$ such that $y=f((x',e'))$.
Thus $f((x',e'))\leqslant_{\Y}f((x,e))$ implies $(x',e')\leqslant (x,e)$ because $f$ is an order embedding under the specialization order.
So $(x',e')\in A$, which leads to $y\in f(A)$. That is, $f(A)$ is a lower subset of $\Y$.
	\end{proof}
	\noindent{\bf Remark.}
	In particularly, $f(\pp)=\da_{\Y}f(\Max(\pp))$ and $f(\pp\backslash \Max (\pp))$ is a lower set.
	Indeed, $\Y$ is divided into two disjoint parts,  a lower set $f(\pp\backslash \Max (\pp))$ and  an upper set $\ua_{\Y}f(\Max(\pp))$.
	
	\begin{lemma}\label{compact}
		Let  $\Y$ be a $T_0$ space and $P$  a bounded complete algebraic poset. Let a pair $\langle \Y, f\rangle $ satisfy (P1)$\sim$(P3). Then for each $\kk\in Q(\Y)$,  the set $K=f^{-1}(\kk)$ is saturated in $\pp$ and $E_K$ is finite.
	\end{lemma}
	\begin{proof}
		Obviously, $K=f^{-1}(\kk)$ is saturated since $f$ is an order embedding and $\kk$ is saturated.
		
		Suppose $E_K$ is infinite. Then there exists an infinite sequence $\{(e_k,e_k):k\in \n\}$ contained in $E_K$.
		By the definition of $E_K$, there exists an infinite set $\{(x_k,e_k):k\in \n\}$ contained in $K\backslash \Max(\pp)$.
		Let $D_n=\da \{(x_k,e_k):k\geqslant n\}$ for $n\in\n$. Then each  $D_n$ is Scott closed in  $\pp\backslash \Max(\pp)$.
		Thus each $f(D_n)$ is closed in $\Y$ by (P3) and
		$$f((x_n,e_n))\in f(D_n)\cap f(K)\subseteq f(D_n)\cap \kk\neq\emptyset.$$
		In fact, the filtered intersection $\bigcap_{n\in\n}f(D_n)=f(\bigcap_{n\in\n}D_n)=\emptyset$, which contradicts to the compactness of $\kk$.
		So $E_K$ is finite.
	\end{proof}
	
	\noindent{\bf Remark.}
	Note that $K$ is not always compact. So we cannot directly derive that $E_K$ is finite by Lemma \ref{compactset}.
	
	Let $P$ be a  poset. For $e\in \Max(P)$, define $\pp_e=\{(x,e):x\leqslant_{P} e\}$.
	Clearly, $\pp=\bigcup\{\pp_e:e\in  \Max(P)\}$, and if $e\neq d$ in $\Max(P)$, then $\pp_e\cap\pp_d=\emptyset$.
	By Lemma \ref{dcpomodel}, we can obtain the following lemma.
	\begin{lemma}\label{scottclosed}
		Let $P$ be a bounded complete algebraic poset and $A$ a Scott closed subset in $\pp$. If $E\subseteq \Max(\pp)$ and $E\cap A=\emptyset$, then $\bigcup_{(e,e)\in E}(A\cap \pp_e)$ is a Scott closed subset of $\pp\backslash  \Max(\pp)$.
	\end{lemma}
	Similar to Lemma \ref{irr} and Corollary \ref{equation1}, we give some equation characterizations for the irreducible closed subsets of $\Y$ as follows.
	\begin{proposition}\label{irr1}
		Let  $\Y$ be a $T_0$ space and $P$  a bounded complete algebraic poset. Let a pair $\langle \Y, f\rangle $ satisfy (P1)$\sim$(P3), $\Aa\in \Irr(\Y)$ and $\Aa^*=\Aa\cap \ua_{\Y} f(\Max (\pp))$.
		Then the following statements hold:
		\begin{enumerate}[(i)]
			\item if $\Aa^*=\emptyset$, then there exists $(x,e) \in \pp\backslash \Max (\pp)$ such that $\Aa=f(\da (x,e))=\da_{\Y}f((x,e))$;
			\item if $\Aa^*\neq\emptyset$, then $\Aa=\cl_{\Y}(\Aa^*)$ and $\Aa^*$ is irreducible closed in $\ua_{\Y} f(\Max (\pp))$.
		\end{enumerate}
	\end{proposition}
	\begin{proof}
		(i) Suppose $\Aa^*=\emptyset$. Then $\Aa\subseteq f(\pp\backslash\Max(\pp))$ by (P2), and  $\Aa$ is  irreducible closed in $f(\pp)$ by Lemma \ref{irrset}.
		Thus $f^{-1}(\Aa)$  is irreducible  Scott closed in $\pp$ because $f:\Sigma\pp\longrightarrow f(\pp)$ is  a homeomorphism.
		In fact,  $$\emptyset=f^{-1}(\Aa^*)=f^{-1}(\Aa\cap \ua_{\Y} f(\Max (\pp)))\supseteq f^{-1}(\Aa\cap f(\Max (\pp)))=f^{-1}(\Aa)\cap \Max (\pp).$$
		By Lemma \ref{irr}, there exists $(x,e) \in \pp\backslash \Max (\pp)$ such that $f^{-1}(\Aa)=\da (x,e)$.
		So $\Aa=f(\da (x,e))$.
		Obviously, %$f(\da (x,e))\subseteq\da_{\Y}f((x,e))$ by the continuity of $f$.
		$f(\da (x,e))$ is closed in $\Y$ by (P2), and hence, $f(\da (x,e))=\cl_{\Y}(f(\da (x,e)))=\da_{\Y}f((x,e))$.
		
		(ii) Let $A=f^{-1}(\Aa)$. Then $A$ is Scott closed in $\pp$, but not always irreducible.
		We have
		\begin{align*}
			A & =\bigcup\{A\cap\pp_e: (e,e)\in A^*\}\cup\bigcup\{A\cap\pp_e: (e,e)\notin A^*\}\\
			&\subseteq \da A^*\cup\bigcup\{A\cap\pp_e: (e,e)\notin A^*\}\\
			&\subseteq \cl_{\pp}(A^*)\cup\bigcup\{A\cap\pp_e: (e,e)\notin A^*\},
		\end{align*}
		where $A^*=A\cap \Max (\pp)$.
		Denote $C=\bigcup\{A\cap\pp_e: (e,e)\notin A^*\}$. Then by Lemma \ref{scottclosed}, $C$ is Scott closed and contained in $\pp\backslash  \Max(\pp)$.
		It follows that $f(C)$ is closed in $\Y$ by (P3).
		In addition, by Lemma \ref{inject}, we have
		\begin{align*}
			f(A^*)&= f(A\cap \Max (\pp)) \\
			&=f(A)\cap f(\Max (\pp))\\
			%    &\subseteq\Aa \cap f(\Max (\pp))\\
			&\subseteq\Aa\cap \ua_{\Y} f(\Max (\pp))\\
			&=\Aa^*.
		\end{align*}
		Hence, $f(\cl_{\pp}(A^*))\subseteq\cl_{\Y}(f(A^*))\subseteq\cl_{\Y}(\Aa^*)$.
		
		Since $f$ is a topological embedding, $f(A)=\Aa\cap f(\pp)$. Then by (P2),  we have
		\begin{align*}
			\Aa & =(\Aa\cap\ua_{\Y}f(\Max(\pp)) )\cup(\Aa\cap f(\pp))\\
			& =\Aa^*\cup f(A)\\
			&\subseteq \cl_{\Y}(\Aa^*)\cup f(\cl_{\pp}(A^*))\cup f(C)\\
			&=\cl_{\Y}(\Aa^*)\cup f(C).
		\end{align*}
		Thus $\Aa=\cl_{\Y}(\Aa^*)$ since $\Aa$ is irreducible and $\Aa^*$ is nonempty.
		By Lemma \ref{irrset},  $\Aa^*$ is irreducible closed in $\ua_{\Y} f(\Max (\pp))$.
		
		The proof is complete.
	\end{proof}
	
	According to Proposition \ref{irr1}, Lemmas  \ref{irrset} and  \ref{closure}, we can obtain the following corollary.
	\begin{corollary}
		Let  $\Y$ be a $T_0$ space and $P$  a bounded complete algebraic poset. Let a pair $\langle \Y, f\rangle $ satisfy (P1)$\sim$(P3). We have the following equations:
		\begin{equation*}\left\{
			\begin{aligned}
				\Irr (\Y) &= \{\cl_{\pp}(A):A\in \Irr (\ua_{\Y} f(\Max (\pp)))\}\cup\{f(\da (x,e)): (x,e)\in \pp\backslash \Max (\pp)\},\\
				\Irr (\ua_{\Y} f(\Max (\pp))) &= \{A\cap\ua_{\Y} f(\Max (\pp)): A\cap\ua_{\Y} f(\Max (\pp))\neq\emptyset ~\rm {and}~ A\in \Irr (\Y) \}.
			\end{aligned}
			\right.\end{equation*}
	\end{corollary}
Next, we will try to get some equation characterizations for the  closed KF-subsets of $\Y$.
	\begin{lemma}\label{cofinal}
		Let $X$ be a $T_0$ space and $\K$ a filtered family of $Q(X)$. If $\K_1$ is a cofinal subfamily of $\K$ and $A\subseteq X$, then $\overline{A}\in m(\K)$ if and only if $\overline{A}\in m(\K_1)$.
	\end{lemma}
	\begin{proof}
		Suppose that $\overline{A}\in m(\K)$. Then $ \overline{A}$ intersects all members of $\K_1$.
		Assume that $B$ is a closed subset of $\overline{A}$ and intersects all members of $\K_1$.
		Then for any $K\in \K$,  there is a $K'\in \K_1$ such that $K'\subseteq K$, which implies that $\emptyset \neq B\cap K'\subseteq B\cap K$.
		It follows that $B=\overline{A}$.
		
		Conversely,  suppose that $\overline{A}\in m(\K_1)$.
		For any $K\in \K$,  there is a $K'\in \K_1$ such that $K'\subseteq K$, which leads to $\emptyset \neq A\cap K'\subseteq A\cap K$.
		That is,  $ \overline{A}$ intersects all members of $\K$.
		Assume that $B$ is a closed subset of $\overline{A}$ and intersects all members of $\K$.
		Then $B$ intersects all members of $\K_1$.
		It follows that $B=\overline{A}$.
	\end{proof}

	\begin{proposition}\label{kfset}
		Let  $\Y$ be a $T_0$ space and $P$  a bounded complete algebraic poset. Let a pair $\langle \Y, f\rangle $ satisfy (P1)$\sim$(P3),  $\Aa\in\KF(\Y)$ and $\Aa^*=\Aa\cap \ua_{\Y} f(\Max (\pp))$.
		Then the following properties hold:
		\begin{enumerate}[(i)]
			\item if $\Aa^*=\emptyset$, then there exists $(x,e) \in \pp\backslash \Max (\pp)$ such that $\Aa=f(\da (x,e))=\da_{\Y}f((x,e))$;
			\item if $\Aa^*\neq\emptyset$, then $\Aa=\cl_{\Y}(\Aa^*)$ and $\Aa^*$ is  a closed KF-set of $\ua_{\Y} f(\Max (\pp))$.
		\end{enumerate}
	\end{proposition}
	\begin{proof}
		By Lemma \ref{subset}, $\Aa$ is irreducible. Then by Proposition \ref{irr1}, we only need to prove that if $\Aa^*\neq\emptyset$, then $\Aa^*$ is  a closed KF-subset of $\ua_{\Y} f(\Max (\pp))$.
		
		Since $\Aa$ is a  KF-subset, there exists a filtered family $\K$ of $Q(\Y)$ such that $\Aa\in m(\K)$.
		Consider the family $\{E_{f^{-1}(\kk)}:\kk\in\K\}$.
		From Lemma \ref{compact}, each $f^{-1}(\kk)$ is saturated and $E_{f^{-1}(\kk)}$ is finite.
		Thus $\{E_{f^{-1}(\kk)}:\kk\in\K\}$ is a filtered family of finite subsets, which has a least element $E_{f^{-1}(\kk_0)}$ for some $\kk_0\in \K$.
		Then the family $$\K_1=\{\kk\subseteq\kk_0:\kk\in \K\}$$ is cofinal; and hence, $\Aa\in m(\K_1)$ by Lemma \ref{cofinal}.
		
		Let $K_0=f^{-1}(\kk_0)$  and $A=f^{-1}(\Aa)$.
		Then $A$ is Scott closed in $\pp$ and $$E_{f^{-1}(\kk)}=E_{K_0}\subseteq f^{-1}(\kk)$$ for all $\kk\in \K_1$.
		To complete the proof, we need to consider the following two cases: $A\cap E_{K_0}\neq\emptyset$ and $A\cap E_{K_0}=\emptyset$.
		
		Case 1. $A\cap E_{K_0}\neq\emptyset$.
		
		Pick an element $(e,e)\in A\cap E_{K_0}\subseteq\Max(\pp)$.
		Then $$f((e,e))\in f(A\cap E_{K_0})=f(A)\cap f(E_{K_0})\subseteq \Aa\cap \kk$$
		for all $\kk\in \K_1$. That is, $\da_{\Y}f((e,e))$ intersects all members of $\K_1$ and is a closed subset of $\Aa$.
		It follows that $\Aa=\da_{\Y}f((e,e))$.
		
		Case 2. $A\cap E_{K_0}=\emptyset$.
		
		Let $C= \bigcup_{(e,e)\in E_{K_0}}(A\cap\pp_e).$ By Lemma \ref{scottclosed}, $C$ is Scott closed in $\pp\backslash \Max(\pp)$, and then $f(C)$ is closed in $\Y$ by (P3).
		Note that $f(C)\subseteq f(\pp\backslash \Max(\pp))$ and $\Aa^*\subseteq\ua_{\Y} f(\Max (\pp))$.
		Thus $f(C)\cap \Aa^*=\emptyset$, which means that $f(C)$ is a proper closed subset  of $\Aa$.
		By the minimality of $\Aa$, there exists a $\kk_1\in \K_1$ such that $f(C)\cap \kk_1=\emptyset$.
		So the family $$\K_2=\{\kk\subseteq\kk_1:\kk\in \K_1\}$$ is cofinal, and then $\Aa\in m(\K_2)$ by Lemma \ref{cofinal}.
		
		For each $\kk\in\K_2$, $f^{-1}(\kk)$ is saturated and $E_{f^{-1}(\kk)}$ is equal to $E_{K_0}$. 
		The following decomposition is natural:
		$$f^{-1}(\kk)=(f^{-1}(\kk)\backslash \da E_{K_0})\cup (f^{-1}(\kk)\cap \da E_{K_0}).$$ 
		Note that
		\begin{align*}
			f^{-1}(\kk)\backslash \da E_{K_0}& \subseteq \Max(\pp), \\
			f^{-1}(\kk)\cap \da E_{K_0} & =\bigcup_{(e,e)\in E_{K_0}}(f^{-1}(\kk)\cap\pp_e).
		\end{align*}
		Then we have $$A\cap f^{-1}(\kk)\cap \da E_{K_0}=f^{-1}(\kk)\cap C\subseteq\pp\backslash \Max(\pp).$$
		So $$f(A\cap f^{-1}(\kk)\cap \da E_{K_0})\subseteq f(f^{-1}(\kk)\cap C)\subseteq\kk\cap f(C)=\emptyset.$$
		It follows that
		\begin{align*}
			\Aa\cap\kk\cap f(\pp) & =f(f^{-1}(\Aa\cap \kk))\\
			&=f(A\cap f^{-1}(\kk))\\
			%  &=f((A\cap (f^{-1}(\kk)\backslash \da E_{K_0})\cup(A\cap f^{-1}(\kk)\cap \da E_{K_0}))\\
			& =f(A\cap (f^{-1}(\kk)\backslash \da E_{K_0}))\cup f(A\cap f^{-1}(\kk)\cap \da E_{K_0})\\
			&=f(A\cap (f^{-1}(\kk)\backslash \da E_{K_0}))\\
			&\subseteq f(\Max(\pp)).
		\end{align*}
		Therefore, $$\Aa\cap \kk=\Aa\cap \kk \cap \ua_{\Y} f(\Max (\pp))=\Aa^*\cap\kk\subseteq \ua_{\Y} f(\Max (\pp)).$$
		Since $\Aa$ is closed and $\kk$ is compact, the intersection $\Aa\cap \kk=\Aa^*\cap\kk$ is compact in $\Y$ as well as in $\ua_{\Y} f(\Max (\pp))$.
		So the family $$\K_3=\{\ua_{\Y}(\Aa^*\cap\kk):\kk\in \K_2\}$$ is a filtered  family of compact saturated subsets of $\ua_{\Y} f(\Max (\pp))$ and $\Aa^*$ intersects all members of $\K_3$.
		
		Suppose that $\Bb$ is a closed subset of $\Aa^*$  and  intersects all members of $\K_3$.
		Pick $y\in \Bb\cap \ua_{\Y}(\Aa^*\cap\kk)$. Then there exists a $y'\in \Aa^*\cap\kk$ such that $y'\leqslant_{\Y}y$.
		Thus $y'\in \Bb$, which means $y'\in\Bb\cap\Aa^*\subseteq\Bb\cap\Aa$.
		It follows that $\cl_{\Y}(\Bb)$ is a closed subset of $\Aa$ and intersects all members of $\K_2$.
		So $\cl_{\Y}(\Bb)=\Aa$, which implies that $\Bb=\Aa^*$.
		
		Hence, $\Aa^*$ is  a closed KF-subset of $\ua_{\Y} f(\Max (\pp))$.
		
		The proof is complete.
	\end{proof}
	
	According to Proposition \ref{kfset}, Lemmas \ref{map} and  \ref{closure}, we can obtain the following corollary.
	\begin{corollary}\label{equation2}
		Let  $\Y$ be a $T_0$ space and $P$  a bounded complete algebraic poset. Let a pair $\langle \Y, f\rangle $ satisfy (P1)$\sim$(P3). We have the following equations:
		\begin{equation*}\left\{
			\begin{aligned}
				\KF (\Y) =& \{\cl_{\Y}(A):A\in \KF (\ua_{\Y} f(\Max (\pp)))\}\cup\{f(\da (x,e)): (x,e)\in \pp\backslash \Max (\pp)\}, \\
				\KF (\ua_{\Y} f(\Max (\pp)))= & \{A\cap\ua_{\Y} f(\Max (\pp)): A\cap\ua_{\Y} f(\Max (\pp))\neq\emptyset ~\rm {and}~ A\in \KF (\Y) \}.
			\end{aligned}
			\right.\end{equation*}
	\end{corollary}
	Obviously, the pair  $\langle \Sigma \pp, \rm{ id}_{\pp}\rangle $ satisfies (P1)$\sim$(P3), where $\rm{ id}_{\pp}$ is an identity mapping on $\pp$.
 So we can get the result as follows.
	\begin{corollary}\label{kfset2}
		Let $P$ be a bounded complete algebraic poset. We have the following equations:
		\begin{equation*}\left\{
			\begin{aligned}
				\KF (\pp) =& \{\cl_{\pp}(A):A\in \KF ( \Max (\pp))\}\cup\{(\da (x,e): (x,e)\in \pp\backslash \Max (\pp)\}, \\
				\KF ( \Max (\pp))= & \{A\cap \Max (\pp): A\cap \Max (\pp)\neq\emptyset ~\rm {and}~ A\in \KF (\pp) \}.
			\end{aligned}
			\right.\end{equation*}
	\end{corollary}
The above corollary characterizes the closed KF-subsets of a $T_1$ space and its Xi-Zhao  model.
%So we can directly formulate  irreducible closed subsets and closed KF-sets.
%
How about the closed WD subsets?

Motivated by \cite{xu_t0_2020}, the well-filtered reflections of $T_0$ spaces can be constructed by  closed WD subsets.
We try  to study  the well-filtered reflections of the Xi-Zhao dcpo model of  a $T_1$ space.

	\begin{definition}(\cite{keimel_d-completions_2009})
		Let $X$ be a $T_0$ space. A \emph{sobrification} of $X$ is a pair $\langle \widetilde{X}, \mu\rangle$ consisting of a sober space $\widetilde{X}$ and a continuous mapping $\mu: X\longrightarrow \widetilde{X}$ satisfying that for any continuous mapping $f: X\longrightarrow Y$ to a sober space, there exists a unique continuous mapping $\tilde{f}: X\longrightarrow \widetilde{X}$ such that $\tilde{f} \circ \mu=f$, that is, the following diagram commutes.
		\begin{figure}[H]
			\centering
			\tikzstyle{format}=[rectangle,draw=white,thin,fill=white]
			%定义语句块的颜色,形状和边
			\tikzstyle{test}=[diamond,aspect=2,draw,thin]
			%定义条件块的形状,颜色
			\tikzstyle{point}=[coordinate,on grid,]
			\begin{tikzpicture}
				\node[format] (X){$X$};
				\node[format,right of=X,node distance=20mm] (Xs){$\widetilde{X}$};
				\node[format,below of=Xs,node distance=15mm] (Y){$Y$};
				\draw[->] (X.east)--node[above]{$\mu$}(Xs.west);
				\draw[->] (X)--node[below]{$f$}(Y);
				\draw[dotted,->] (Xs)--node[right]{$\tilde{f}$}(Y);
			\end{tikzpicture}
			%\caption{A commutative diagram}
			%\label{fig1}
		\end{figure}
	\end{definition}
	\begin{lemma}(\cite{goubault-larrecq_non-hausdorff_2013})\label{completion1}
		Let $X$ be a $T_0$ space.
		The space $X^s=P_H(\Irr(X))$ with the canonical mapping $\eta_{X} : X \longrightarrow X^s$ is the sobrification of $X$.
	\end{lemma}

\begin{definition}(\cite{shen_well-filtered_2019})
	Let $X$ be a $T_0$ space. A \emph{well-filtered reflection} of $X$ is a pair $\langle \widetilde{X}, \mu\rangle$ consisting of a well-filtered space $\widetilde{X}$ and a continuous mapping $\mu: X\longrightarrow \widetilde{X}$ satisfying that for any continuous mapping $f: X\longrightarrow Y$ to a well-filtered space, there exists a unique continuous mapping $\tilde{f}: X\longrightarrow \widetilde{X}$ such that $\tilde{f} \circ \mu=f$, that is, the following diagram commutes.
	\begin{figure}[H]
		\centering
		\tikzstyle{format}=[rectangle,draw=white,thin,fill=white]
		%定义语句块的颜色,形状和边
		\tikzstyle{test}=[diamond,aspect=2,draw,thin]
		%定义条件块的形状,颜色
		\tikzstyle{point}=[coordinate,on grid,]
		\begin{tikzpicture}
			\node[format] (X){$X$};
			\node[format,right of=X,node distance=20mm] (Xs){$\widetilde{X}$};
			\node[format,below of=Xs,node distance=15mm] (Y){$Y$};
			\draw[->] (X.east)--node[above]{$\mu$}(Xs.west);
			\draw[->] (X)--node[below]{$f$}(Y);
			\draw[dotted,->] (Xs)--node[right]{$\tilde{f}$}(Y);
		\end{tikzpicture}
		%\caption{A commutative diagram}
		%\label{fig1}
	\end{figure}
\end{definition}
There are two well-filtered reflections of a $T_0$ space which are respectively constructed  in \cite{shen_well-filtered_2019,xu_t0_2020}.
\begin{lemma}(\cite{shen_well-filtered_2019})\label{shen}
	For any $T_0$ space $Z_0$, let $Z$ be a sober space that has $Z_0$ as a subspace. For each ordinal $\beta$, define
	\begin{enumerate}[(i)]
		\item $Z_{\beta+1}=\{z\in Z: \exists F\in \KF(Z_{\beta}), \cl_{Z}(F)=\da_{Z}z\}$;
		\item $Z_{\beta}=\bigcup_{\gamma<\beta}Z_{\gamma}$ for a limit ordinal $\beta$.
	\end{enumerate}
	Then there exists an ordinal $\alpha$ such that $Z_{\alpha}=Z_{\alpha+1}$ and $\{Z_{\beta}:\beta\leqslant \alpha\}$ is an increasing transfinite sequence. In addition, $Z_{\alpha}$ is the well-filtered reflection of $Z_0$.
\end{lemma}
	\begin{lemma}(\cite{xu_t0_2020})\label{completion2}
		Let $X$ be a $T_0$ space.
		The space $X^w=P_H(\WD(X))$ with the canonical mapping $\eta_{X} : X \longrightarrow X^w$ is the well-filtered reflection of $X$.
	\end{lemma}	
	
	Let $X$ be a $T_1$ space and $\pp$ the Xi-Zhao model of $X$.
We will construct the sobrifications and the well-filtered reflections of $\Max(\pp)$ and  $\pp$.

By Lemma \ref{completion1}, the space $\Max(\pp)^s=P_H(\Irr(\Max(\pp)))$ with the canonical mapping $\eta_{\Max(\pp)} : \Max(\pp) \longrightarrow \Max(\pp)^s$ is the sobrification of $X$.
	The space ${\pp}^s=P_H(\Irr(\pp))$ with the canonical mapping $\eta_{\pp} : \Sigma\pp \longrightarrow \pp^s$ is the sobrification of $\Sigma \pp$.
	By Lemma \ref{embedding}, the space $P_H(\Irr(\pp))$ has the topology $\{\Diamond_{\pp^s}U:U\in \sigma(\pp)\}$, where $$\Diamond_{\pp^s}U=\{A\in \Irr(\pp):A\cap U\neq\emptyset\},$$ and the closed subsets are exactly the set of forms $\Box_{\pp^s}A=\da_{\pp^s}A$ for a Scott closed subset $A$ of $\pp$.
	Similarly, the space $P_H(\Irr(\Max(\pp)))$ has the topology $\{\Diamond_{\Max(\pp)^s}(U\cap\Max(\pp)):~U\in \sigma(\pp)\}$, where $$\Diamond_{\Max(\pp)^s}(U\cap\Max(\pp))=\{A\in \Irr(\Max(\pp)):A\cap U\neq\emptyset\},$$  and the closed subsets are exactly the set of forms $\Box_{\Max(\pp)^s}(A\cap\Max(\pp))=\da_{\Max(\pp)^s}(A\cap\Max(\pp))$ for a Scott closed subset $A$ of $\pp$.
	
	Obviously, the inclusion mapping $i: \Max(\pp)\longrightarrow \Sigma\pp$ is continuous. We define a mapping $j: P_H(\Irr(\Max(\pp)))\longrightarrow P_H(\Irr(\pp))$ by
	$$\forall A\in \Irr(\Max(\pp)), j(A)=\cl_{\pp}(A).$$
	By Lemma \ref{map}, $j$ is well-defined.
	There are some properties of the map $j$ as follows.
\begin{proposition}\label{prop:embed}
  Let $P$ be a bounded complete algebraic poset and the map $j: \Max(\pp)^s \longrightarrow \pp^s$ by $j(A)=\cl_{\pp}(A)$.
  Then the following statements hold.
  \begin{itemize}
    \item[(1)] $j$ is a topological embedding;\label{j1}
    \item[(2)] $j\circ \eta_{\Max(\pp)}=\eta_{\pp}\circ i$, that is, the following diagram  commutes;\label{j2}
	\begin{figure}[htbp]
		\centering
		\tikzstyle{format}=[rectangle,draw=white,thin,fill=white]		
		\tikzstyle{test}=[diamond,aspect=2,draw,thin]
		\tikzstyle{point}=[coordinate,on grid,]
		\begin{tikzpicture}
			\node[format] (maxp){$\Max(\pp)$};
			\node[format,right of=maxp,node distance=40mm] (maxps){$P_H(\Irr(\Max(\pp)))$};
			\node[format,below of=maxp,node distance=20mm] (P){$\Sigma\pp$};
			\node[format,right of=P,node distance=40mm] (Ps){$P_H(\Irr(\pp))$};
			\draw[->] (maxp.east)--node[above]{$\eta_{\Max(\pp)}$}(maxps.west);
			\draw[->] (maxp)--node[left]{$i$}(P);
			\draw[->] (P.east)--node[below]{$\eta_{\pp}$}(Ps.west);
			\draw[->] (maxps)--node[right]{$j$}(Ps);
		\end{tikzpicture}
		%\caption{A commutative diagram of Claim \ref{j2}.}
		% \label{fig1}
	\end{figure}
    \item[(3)] $j(\Irr(\Max(\pp))=\ua_{\Irr(\pp)}\eta_{\pp}(\Max(\pp))$;\label{irr2}
    \item[(4)] The inverse map $j^{-1}: \ua_{\Irr(\pp)}\eta_{\pp}(\Max(\pp))\longrightarrow \Irr(\Max(\pp))$ is given by
		$$j^{-1}(A)=A\cap \Max(\pp).$$
  \end{itemize}
\end{proposition}
\begin{proof}
  (1) 	For $A_1\neq A_2$ in $\Irr(\Max(\pp))$, we have
	$\cl_{\pp}(A_1)\cap\Max(\pp)\neq \cl_{\pp}(A_2)\cap\Max(\pp).$
	It follows that $j(A_1)\neq j(A_2)$, i.e., $j$ is injective. For $U\in \sigma(\pp)$, we have
	\begin{align*}
		j^{-1}(\Diamond_{\pp^s}U) &=\{A\in\Irr(\Max(\pp)): \cl_{\pp}(A)\cap U\neq \emptyset\} \\
		& =\{A\in\Irr(\Max(\pp)): A\cap U\neq \emptyset\}\\
		&=\Diamond_{\Max(\pp)^s}(U\cap\Max(\pp)).
	\end{align*}
	So $j$ is continuous and almost open.
	%Hence, $j$ is a topological embedding.

(2) 	Let $(e,e)\in\Max(\pp)$. We have
	$$j\circ \eta_{\Max(\pp)}((e,e))=j(\{(e,e)\})=\da (e,e)=\eta_{\pp}((e,e))=\eta_{\pp}\circ i((e,e)).$$

(3) 		By Corollary \ref{equation1}, we have
	\begin{align*}
		j(\Irr(\Max(\pp)) & =\{\cl_{\pp}(A):A\in \Irr (\Max (\pp))\} \\
		& =\{A\in\Irr(\pp):A\cap \Max(\pp)\neq\emptyset\}\\
		& =\{A\in\Irr(\pp): \exists (e,e)\in\Max(\pp), s.t. ~\eta_{\pp}((e,e))=\da(e,e) \subseteq A\}\\
		&=\ua_{\Irr(\pp)}\eta_{\pp}(\Max(\pp)).
	\end{align*}	

(4) It is trivial from Lemma \ref{irr}.
\end{proof}

	Immediately,  we can obtain the following results.
	\begin{corollary}\label{embedsober}
		For a $T_1$ space $X$ and its Xi-Zhao model $\pp$, $X^s$ can be topologically embedded into $\pp^s$ as a saturated subspace.
	\end{corollary}
	\begin{corollary}\label{equation0}
		Let $P$ be a bounded complete algebraic poset. Then
		$$\Irr(\pp)=\ua_{\Irr(\pp)}\eta_{\pp}(\Max(\pp))\cup \eta_{\pp}(\pp\backslash \Max (\pp)).$$
	\end{corollary}
The above equation reveals that the pair $\langle P_H (\Irr(\pp)),\eta_{\pp}\rangle$ satisfies (P3).
Here are some examples that satisfy (P1) $\sim$ (P3).
	
	\begin{lemma}\label{p123}
		Let $P$ be a bounded complete algebraic poset.
		If $\si(\pp)\subseteq \G\subseteq\Irr(\pp)$, then the pair $\langle P_H (\G),\eta_{\pp}\rangle$ satisfies (P1)$\sim$(P3).
	\end{lemma}
	\begin{proof}
		By Lemma \ref{embedding}, the topology of the space $P_H({\G})$ is precisely the relative topology from $P_H(\Irr(\pp))$.
		That is, $\eta_{\pp}:\Sigma\pp\longrightarrow P_H({\G})$ is a topological embedding.
		
		Note that $\eta_{\pp}(\Max(\pp)), \eta_{\pp}(\pp\backslash \Max (\pp))\subseteq \G.$
		By Corollary \ref{equation0}, we have
		\begin{align*}
			\G =\G\cap\Irr(\pp) & = (\ua_{\Irr(\pp)}\eta_{\pp}(\Max(\pp))\cap \G)\cup (\eta_{\pp}(\pp\backslash \Max (\pp))\cap \G)\\
			& =\ua_{\G}\eta_{\pp}(\Max(\pp))\cup \eta_{\pp}(\pp\backslash \Max (\pp)).
		\end{align*}
		Let  $\da (x,e)\in \eta_{\pp}(\pp)$. Suppose that $A$ is a subset of $\da (x,e)$ in $\Irr(\pp)$.
		According to Lemma \ref{irr}, there exists $(y,d)\in \pp$ such that $\da (y,d)\subseteq \da (x,e)$.
		Then $\da (y,d)\in \eta_{\pp}(\pp)$.
		It follows that $\eta_{\pp}(\pp)$ is a lower set in $\Irr(\pp)$.
		So $\eta_{\pp}(\pp)$ is a lower set in $\G$.
		
		Let $C$ be a Scott closed subset of $\pp\backslash \Max (\pp)$.
		Obviously, $C$ is also Scott closed in $\pp$.
		By Lemma \ref{irr}, we have
		\begin{align*}
			\eta_{\pp}(C) &=\{\da (x,e):(x,e)\in C\} \\
			&=\{B\in\Irr(\pp): B\subseteq C \subseteq \pp\backslash \Max (\pp)\}\\
			& =\da_{\pp^s}C.
		\end{align*}
		Note that $\da_{\pp^s}C\subseteq \si(\pp)\subseteq \G\subseteq\Irr(\pp)$.
		It implies that $\eta_{\pp}(C)=\da_{\pp^s}C=\da_{\G}C$.
		So $\eta_{\pp}(C)$ is closed in $P_H(\G)$.
	\end{proof}	

In the following, we will construct the well-filtered reflections of $\Max(\pp)$ and  $\pp$, according to Lemma \ref{shen}.
%	Next, we will prove our main result.
	
%	Let $X$ be a $T_1$ space and $\pp$ the Xi-Zhao model of $X$.
Denote
	\begin{align*}
		X_0 & =\eta_{\Max(\pp)}(\Max(\pp))=\{\{(e,e)\}:e\in\Max(P)\},\\
		Y_0& =\eta_{\pp}(\pp)=\{\da (x,e): (x,e)\in \pp\}.
	\end{align*}
	For each ordinal $\beta$, define
	\begin{enumerate}[(i)]
		\item $X_{\beta+1}=\{A\in \Max(\pp)^s: \exists \A\in \KF(X_{\beta}), ~{\rm s.t.}~ \cl_{\Max(\pp)^s}(\A)=\da_{\Max(\pp)^s} A\}$;
		\item $X_{\beta}=\bigcup_{\gamma<\beta}X_{\gamma}$ for a limit ordinal $\beta$;
		\item $Y_{\beta+1}=\{B\in \pp^s: \exists \B\in \KF(Y_{\beta}), ~{\rm s.t.}~\cl_{\pp^s}(\B)=\da_{\pp^s}B\}$;
		\item $Y_{\beta}=\bigcup_{\gamma<\beta}Y_{\gamma}$ for a limit ordinal $\beta$.
	\end{enumerate}
	
	\begin{claim}\label{p1234}
		For each ordinal $\beta$, the pair $\langle  P_H(Y_{\beta}),\eta_{\pp}\rangle$ satisfies the following properties:
		\begin{itemize}
			\item[(P1)] $\eta_{\pp}:\Sigma\pp\longrightarrow P_H(Y_{\beta})$ is a topological embedding;
			\item[(P2)] $Y_{\beta}=\ua_{Y_{\beta}}\eta_{\pp}(\Max(\pp))\cup \eta_{\pp}(\pp\backslash \Max (\pp))$ and $\eta_{\pp}(\pp)$ is a lower set in $Y_{\beta}$;
			\item[(P3)] if $C$ is Scott closed in $\pp\backslash \Max (\pp)$, then $\eta_{\pp}(C)$ is closed in $P_H(Y_{\beta})$;
		\end{itemize}
	\end{claim}
	Obviously, $\eta_{\pp}(\pp)=Y_0\subseteq Y_{\beta}\subseteq \Irr(\pp)$ by Lemma \ref{shen}.
	Then (P1)$\sim$(P3) hold by Lemma \ref{p123}.
	\begin{claim}\label{embed2}
		For each ordinal $\beta$, $j(X_{\beta})=\ua_{Y_{\beta}}\eta_{\pp}(\Max (\pp))$, where	$\forall A\in \Irr(\Max(\pp)), j(A)=\cl_{\pp}(A).$
	\end{claim}
	
	{\bf Step 1.} We claim that $j(X_{0})=\ua_{Y_{0}}\eta_{\pp}(\Max (\pp))$.
	
	Note that $\eta_{\pp}(\Max (\pp))=\{\da (e,e):(e,e)\in \Max(\pp)\}$ and $Y_0 =\{\da (x,e): (x,e)\in \pp\}.$
	Then $\ua_{Y_{0}}\eta_{\pp}(\Max (\pp))=\eta_{\pp}(\Max (\pp))$.
	So
	\begin{align*}
		j(X_{0})& =\{\da (e,e):e\in \Max(P)\} \\
		& =\ua_{Y_{0}}\eta_{\pp}(\Max (\pp)).
	\end{align*}
	
	{\bf Step 2.} Let $\beta<\alpha$. Assume $j(X_{\beta})=\ua_{Y_{\beta}}\eta_{\pp}(\Max (\pp))\subseteq Y_{\beta}$ (see Figure \ref{fig2}).
	We claim that $j(X_{\beta+1})=\ua_{Y_{\beta+1}}\eta_{\pp}(\Max (\pp))$.
	\begin{figure}[htbp]
		\centering
		\tikzstyle{format}=[rectangle,draw=white,thin,fill=white]
		\tikzstyle{test}=[diamond,aspect=2,draw,thin]
		\tikzstyle{point}=[coordinate,on grid,]
		\begin{tikzpicture}
			%\node[format] (X){$X$};
			%\node[format] (maxp){$\Max(\pp)$};
			%\node[format,below of=maxp,node distance=25mm] (P){$\Sigma\pp$};
			\node[format] (X0){$X_0$};
			\node[format,right of=X0,node distance=30mm] (Xb){$X_{\beta}$};
			\node[format,right of=Xb,node distance=30mm] (Xb+1){$X_{\beta+1}$};
			%\node[format,right of=Xb+1,node distance=20mm] (Xa){$X_{\alpha}$};
			\node[format,below of=X0,node distance=20mm] (Y0){$Y_{0}$};
			\node[format,right of=Y0,node distance=30mm] (Yb){$Y_{\beta}$};
			\node[format,right of=Yb,node distance=30mm] (Yb+1){$Y_{\beta+1}$};
			%\node[format,right of=Yb+1,node distance=20mm] (Ya){$Y_{\alpha}$};
			\node[format,right of=Xb+1,node distance=30mm] (maxps){$\Irr(\Max(\pp))$};
			\node[format,right of=Yb+1,node distance=30mm] (Ps){$\Irr(\pp)$};
			
			%\draw[double] (X0)--(maxp);
			%\draw[double] (Y0)--(P);
			%\draw[->] (maxp)--node[left]{$i$}(P);
			\draw[->] (maxps)--node[right]{$j$}(Ps);
			\draw[->](X0.east)--node[above]{$\subseteq$}(Xb.west);
			\draw[->](Y0.east)--node[above]{$\subseteq$}(Yb.west);
			\draw[->] (X0)--node[right]{$j$}(Y0);
			\draw[->](Xb.east)--node[above]{$\subseteq$}(Xb+1.west);
			\draw[->](Yb.east)--node[above]{$\subseteq$}(Yb+1.west);
			\draw[->] (Xb)--node[right]{$j$}(Yb);
			%\draw[->] (Xb+1)--(Yb+1);
			%\draw[dashed](Xb+1.east)--(Xa.west);
			%\draw[dashed](Yb+1.east)--(Ya.west);
			%\draw[->] (Xa)--(Ya);
			\draw[->](Xb+1.east)--node[above]{$\subseteq$}(maxps.west);
			\draw[->](Yb+1.east)--node[above]{$\subseteq$}(Ps.west);
		\end{tikzpicture}
		\caption{The  ranges of  the mapping $j$ with different domains.}
		\label{fig2}
	\end{figure}
	
	(1)  $j(X_{\beta+1})\subseteq\ua_{Y_{\beta+1}}\eta_{\pp}(\Max (\pp))$.
	
	Recall from Proposition \ref{prop:embed} that  $j(\Irr(\Max(\pp))=\ua_{\Irr(\pp)}\eta_{\pp}(\Max(\pp))$.
	Then $$j(X_{\beta+1})\subseteq j(\Irr(\Max(\pp)) \subseteq \ua_{\Irr(\pp)}\eta_{\pp}(\Max(\pp)).$$
	It suffices to show that  $j(X_{\beta+1})\subseteq Y_{\beta+1}$.
	
	Assume $A\in X_{\beta+1}\subseteq\Irr(\Max(\pp))$.
	Then there exists $\A\in \KF(X_{\beta})$ such that $\cl_{\Max(\pp)^s}(\A)=\da_{\Max(\pp)^s} A$.
	By Lemma \ref{closure}, $\A=\cl_{\Max(\pp)^s}(\A)\cap X_{\beta}$.
	That is, $\A=\da_{\Max(\pp)^s} A\cap X_{\beta}$.
	Indeed, $j: X_{\beta}\longrightarrow Y_{\beta}$ is continuous by hypothesis.
	Thus $\cl_{Y_{\beta}}(j(\A))$ is a closed KF-set of $Y_{\beta}$ by Lemma \ref{map}.
	Hence,
	\begin{align*}
		\cl_{\pp^s}\cl_{Y_{\beta}}(j(\A))
		&=\cl_{\pp^s}(j(\A)) \\
		&=\cl_{\pp^s} (j(\da_{\Max(\pp)^s} A\cap X_{\beta}))\\
		& \subseteq \cl_{\pp^s} (j(\da_{\Max(\pp)^s} A)))\\
		&=\cl_{\pp^s}(j(A))\\
		&=\da_{\pp^s}\cl_{\pp}(A).
	\end{align*}
	
	Conversely, for any $U\in \sigma(\pp)$ with $\cl_{\pp}(A)\in \Diamond_{\pp^s}U$, we have $A\cap U\neq\emptyset$, i.e., $A\cap U\cap\Max(\pp)\neq\emptyset$.
	Pick $(e,e)\in A\cap U\cap\Max(\pp)$.
	Then
	$$\{(e,e)\}\subseteq \da_{\Max(\pp)^s} A\cap X_{0}\subseteq \da_{\Max(\pp)^s} A\cap X_{\beta}.$$
	It follows that
	$$\da (e,e)\in j(\da_{\Max(\pp)^s} A\cap X_{\beta})\cap\Diamond_{\pp^s}U\neq\emptyset.$$
	That is to say, $$\cl_{\pp}(A)\in\cl_{\pp^s} (j(\da_{\Max(\pp)^s} A\cap X_{\beta})).$$
	Hence, $$\da_{\pp^s}\cl_{\pp}(A)=\cl_{\pp^s}\cl_{Y_{\beta}}(j(\A)).$$
	So $j(A)=\cl_{\pp}(A)\in Y_{\beta+1}$.
	Moreover, the mapping $j: X_{\beta+1}\longrightarrow Y_{\beta+1}$ is a topological embedding.
	
	Since  $j(X_{\beta+1})\subseteq Y_{\beta+1}$, we have $j(X_{\beta+1})\subseteq Y_{\beta+1}\cap \ua_{\Irr(\pp)}\eta_{\pp}(\Max(\pp))=\ua_{Y_{\beta+1}}\eta_{\pp}(\Max (\pp))$.

	(2) $j(X_{\beta+1})\supseteq\ua_{Y_{\beta+1}}\eta_{\pp}(\Max (\pp))$.
	
	Let $B\in \ua_{Y_{\beta+1}}\eta_{\pp}(\Max (\pp))\subseteq\ua_{\Irr(\pp)}\eta_{\pp}(\Max (\pp))$.
	Then $B^*=B\cap\Max(\pp)\neq\emptyset$, $B=\cl_{\pp}(B^*)$ and $B^*\in\Irr(\Max (\pp))$ by Lemma \ref{irr}.
	We need to prove that $B^*\in X_{\beta+1}$.
	
	Since $B\in Y_{\beta+1}$, there exists $\B\in\KF(Y_{\beta})$ such that $\cl_{\pp^s}(\B)=\da_{\pp^s}B$.
	By Lemma \ref{closure}, $\B=\cl_{\pp^s}(\B)\cap Y_{\beta}$.
	Thus \begin{align*}
		\B\cap \ua_{Y_{\beta}}\eta_{\pp}(\Max (\pp)) & =\cl_{\pp^s}(\B)\cap Y_{\beta}\cap \ua_{Y_{\beta}}\eta_{\pp}(\Max (\pp)) \\
		& =\da_{\pp^s}B\cap \ua_{Y_{\beta}}\eta_{\pp}(\Max (\pp)).
	\end{align*}
	Pick an element $(e,e)$ from $B^*$.
	In fact, $\da(e,e)\in \da_{\pp^s}B$ and  $\da(e,e)\in\eta_{\pp}(\Max (\pp))\subseteq\ua_{Y_{\beta}}\eta_{\pp}(\Max (\pp))$.
	That is,  $\B\cap \ua_{Y_{\beta}}\eta_{\pp}(\Max (\pp))$ is nonempty.
	By Proposition \ref{kfset}, the intersection $\B\cap \ua_{Y_{\beta}}\eta_{\pp}(\Max (\pp))$ is a closed KF-subset of  $\ua_{Y_{\beta}}\eta_{\pp}(\Max (\pp))$ because $\B$ is a closed KF-subset of $Y_{\beta}$.

	By assumption, the mapping $j: X_{\beta}\longrightarrow \ua_{Y_{\beta}}\eta_{\pp}(\Max (\pp))$ is a homeomorphism.
	It suggests that
	$$j^{-1}(\B\cap \ua_{Y_{\beta}}\eta_{\pp}(\Max (\pp)))=j^{-1}(\da_{\pp^s}B\cap \ua_{Y_{\beta}}\eta_{\pp}(\Max (\pp)))\in \KF(X_{\beta}).$$
	So
	\begin{align*}
		& \cl_{\Max(\pp)^s}\left(j^{-1}(\da_{\pp^s}B\cap \ua_{Y_{\beta}}\eta_{\pp}(\Max (\pp)))\right) \\
		\subseteq~& \cl_{\Max(\pp)^s}\left(j^{-1}(\da_{\pp^s}B)\right)\\
		%  \subseteq~& \cl_{\Max(\pp)^s}(j^{-1}(B))\\
		=~&\cl_{\Max(\pp)^s}(j^{-1}(B))\\
		=~&\da_{\Max(\pp)^s}B^*.
	\end{align*}
	
	Conversely, for any $U\in \sigma(\pp)$ with $B^*\in \Diamond_{\Max(\pp)^s}(U\cap\Max(\pp)) $, we have $B^*\cap U\cap\Max(\pp)\neq\emptyset$.
	Pick an element $(e,e)\in B^*\cap U\cap\Max(\pp)=B\cap U\cap\Max(\pp)$.
	Then the set $\da (e,e)\in \da_{\pp^s}B\cap \ua_{Y_{\beta}}\eta_{\pp}(\Max (\pp)))$, which implies that
	$$\{(e,e)\}=j^{-1}(\da (e,e))\in j^{-1}(\da_{\pp^s}B\cap \ua_{Y_{\beta}}\eta_{\pp}(\Max (\pp))).$$
	Thus,
	$$j^{-1}(\da_{\pp^s}B\cap \ua_{Y_{\beta}}\eta_{\pp}(\Max (\pp)))\cap\Diamond_{\Max(\pp)^s}(U\cap\Max(\pp))\neq\emptyset. $$
	That is to say,
	$$B^*\in \cl_{\Max(\pp)^s}\left(j^{-1}(\da_{\pp^s}B\cap \ua_{Y_{\beta}}\eta_{\pp}(\Max (\pp)))\right).$$
	Hence,
	$$\da_{\Max(\pp)^s}B^*= \cl_{\Max(\pp)^s}\left(j^{-1}(\da_{\pp^s}B\cap \ua_{Y_{\beta}}\eta_{\pp}(\Max (\pp)))\right).$$
	Therefore $B^*\in X_{\beta+1}.$
	
	%By Lemma \ref{irr}, $j(B^*)=\cl_{\pp}(B^*)=B$. We have $j(X_{\beta+1})\supseteq\ua_{Y_{\beta+1}}\eta_{\pp}(\Max (\pp))$.
	
	We conclude that $j(X_{\beta+1})=\ua_{Y_{\beta+1}}\eta_{\pp}(\Max (\pp))$.
	
	{\bf Step 3.} Let $\beta\leqslant\alpha$ be a limit ordinal. Assume for any $\gamma<\beta$, $j(X_{\gamma})=\ua_{Y_{\gamma}}\eta_{\pp}(\Max (\pp))$.
	We claim that $j(X_{\beta})=\ua_{Y_{\beta}}\eta_{\pp}(\Max (\pp))$.
	
	We have
	\begin{align*}
		j(X_{\beta}) & = j(\bigcup_{\gamma <\beta}X_{\gamma})\\
		& = \bigcup_{\gamma <\beta}j(X_{\gamma})\\
		&= \bigcup_{\gamma <\beta}\ua_{Y_{\gamma}}\eta_{\pp}(\Max (\pp))\\
		&=\ua_{Y_{\beta}}\eta_{\pp}(\Max (\pp)).
	\end{align*}
	
	\begin{claim}\label{equal}
		There exists an ordinal $\alpha$ such that $X_{\alpha}=X_{\alpha+1}$ and $Y_{\alpha}=Y_{\alpha+1}$.
		In this case, $X_{\alpha}=\WD(\Max(\pp))$ and $Y_{\alpha}=\WD(\pp)$.
	\end{claim}
	
	By Lemma \ref{shen}, there exists $\alpha$ such that $X_{\alpha}=X_{\alpha+1}$.
	Thus
	\begin{align*}
		Y_{\alpha} & = \ua_{Y_{\alpha}}\eta_{\pp}(\Max(\pp))\cup \eta_{\pp}(\pp\backslash \Max (\pp))&({\rm~by ~(P2)})\\
		& =j(X_{\alpha})\cup \eta_{\pp}(\pp\backslash \Max (\pp))&({\rm~by ~Claim ~\ref{embed2}})\\
		&=j(X_{\alpha+1})\cup \eta_{\pp}(\pp\backslash \Max (\pp)) &({\rm~since ~}X_{\alpha}=X_{\alpha+1})\\
		&= \ua_{Y_{\alpha+1}}\eta_{\pp}(\Max(\pp))\cup \eta_{\pp}(\pp\backslash \Max (\pp)) &({\rm~by ~Claim ~\ref{embed2}})\\
		&= Y_{\alpha+1}.&({\rm~by ~(P2)})
	\end{align*}
	According to Lemmas \ref{shen} and \ref{completion2},  $\langle  P_H(X_{\alpha}),\eta_{\Max(\pp)}\rangle$ and $\langle  P_H(\WD(\Max(\pp))),\eta_{\Max(\pp)}\rangle$ are both well-filterifications of the maximal point space $\Max(\pp)$.
	%In addition, $X_{\alpha}$ and $\WD(\Max(\pp))$ are also contained in $\Irr(\Max(\pp))$.
	From the definition of well-filterification, there exists a unique mapping $h:X_{\alpha}\longrightarrow \WD(\Max(\pp))$ such that $\eta_{\Max(\pp)}=h\circ\eta_{\Max(\pp)}$ and a unique mapping $g:\WD(\Max(\pp))\longrightarrow X_{\alpha}$ such that $\eta_{\Max(\pp)}=g\circ\eta_{\Max(\pp)}$,
	that is, the following diagram commutes.
	\begin{figure}[H]
		\centering
		\tikzstyle{format}=[rectangle,draw=white,thin,fill=white]
		\tikzstyle{test}=[diamond,aspect=2,draw,thin]
		\tikzstyle{point}=[coordinate,on grid,]
		\begin{tikzpicture}
			\node[format] (maxp){$\Max (\pp)$};
			\node[format,right of=maxp,node distance=30mm] (Xs){$X_{\alpha}$};
			\node[format,below of=Xs,node distance=13mm] (Y){$\WD(\Max(\pp))$};
			\draw[->] (maxp.east)--node[yshift=2mm]{\footnotesize$\eta_{\Max(\pp)}$}(Xs.west);
			\draw[->] (maxp)--node[yshift=-3mm]{\footnotesize$\eta_{\Max(\pp)}$}(Y);
			\draw[dotted,->] ( [xshift=1mm]Xs.south)--node[right]{$h$}( [xshift=1mm]Y.north);
			\draw[dotted,->]  ( [xshift=-1mm]Y.north)--node[left]{$g$}([xshift=-1mm]Xs.south);
		\end{tikzpicture}
		\caption{A commutative diagram of Claim \ref{equal}.}
		%\label{fig1}
	\end{figure}

	Hence, $X_{\alpha}=\WD(\Max(\pp))$.
	Similarly, $Y_{\alpha}=\WD(\pp)$.
	
	Immediately,  we can obtain the following results.
	\begin{theorem}\label{embedwf}
		For a $T_1$ space $X$ and its Xi-Zhao model $\pp$, $X^w$ can be topologically embedded into $\pp^w$ as a saturated subspace.
		In addition, the following diagram commutes.	
		\begin{figure}[htbp]
			\centering
			\tikzstyle{format}=[rectangle,draw=white,thin,fill=white]
			\tikzstyle{test}=[diamond,aspect=2,draw,thin]
			\tikzstyle{point}=[coordinate,on grid,]
			\begin{tikzpicture}
				%\node[format] (X){$X$};
				\node[format] (maxp){$\Max(\pp)$};
				\node[format,right of=maxp,node distance=40mm] (maxps){$P_H(\WD(\Max(\pp)))$};
				\node[format,below of=maxp,node distance=20mm] (P){$\Sigma\pp$};
				\node[format,right of=P,node distance=40mm] (Ps){$P_H(\WD(\pp))$};
				
				%\draw[dashed] (X)--(maxp);
				\draw[->] (maxp.east)--node[above]{$\eta_{\Max(\pp)}$}(maxps.west);
				\draw[->] (maxp)--node[left]{$i$}(P);
				\draw[->] (P.east)--node[below]{$\eta_{\pp}$}(Ps.west);
				\draw[->] (maxps)--node[right]{$j$}(Ps);
				%\draw[->] ([yshift=1mm]sober.east)--([yshift=1mm]wf.west);
				%\draw[->,dashed] ([yshift=1mm]wwf.west)--node[above]{\footnotesize +}([yshift=1mm]ws.east);
				%\draw[->] ([yshift=-1mm]ws.east)--([yshift=-1mm]wwf.west);
				%\draw[->] (wf)--(wwf);
				%\draw[->] ([yshift=1mm]wf.east)--([yshift=1mm]owf.west);
				%\draw[->] (wwf)--(T0);
				%\draw[->] (owf)--(T0);
				%\draw[->,dashed] (9,-2.3)--(9,-2.8)--node[below]{\footnotesize +}(2,-2.8)--(2,-2.3);
			\end{tikzpicture}
			 \caption{A commutative diagram of Theorem \ref{embedwf}.}
			\label{fig3}
		\end{figure}
	\end{theorem}	
	Like irreducible closed subsets and  closed KF-subsets, we can also give some characterizations for the closed WD subsets below.
	\begin{corollary}\label{equation3}
		Let $P$ be a bounded complete algebraic poset. Let $A\in\WD(\pp)$ and $A^*=A\cap \Max (\pp)$.
		Then the following properties hold:
		\begin{enumerate}[(i)]
			\item if $A^*=\emptyset$, then there exists $(x,e) \in \pp\backslash \Max (\pp)$ such that $A=\da (x,e)$;
			\item if $A^*\neq\emptyset$, then $A=\cl_{\pp}(A^*)$ and $A^*$ is  a closed WD subset of $\Max (\pp)$.
		\end{enumerate}
		Thus  we have the following equations:
		\begin{equation*}\left\{
			\begin{aligned}
				\WD (\pp) =& \{\cl_{\pp}(B):A\in \WD ( \Max (\pp))\}\cup\{(\da (x,e): (x,e)\in \pp\backslash \Max (\pp)\}, \\
				\WD ( \Max (\pp))= & \{B\cap \Max (\pp): B\cap \Max (\pp)\neq\emptyset {~\rm and~} B\in \WD (\pp) \}.
			\end{aligned}
			\right.\end{equation*}
	\end{corollary}
	\begin{proof}
		Since every WD set is irreducible, we only consider the case that $A^*\neq\emptyset$ by Lemma \ref{irr}.
		According to Claims \ref{embed2} and \ref{equal}, we have
		% and Claim \ref{equal},
		%$$\WD(\pp)=\ua_{\WD(\pp)}\eta_{\pp}(\Max(\pp))\cup \eta_{\pp}(\pp\backslash \Max (\pp)).$$
		%By Claim \ref{embed2} and Claim \ref{equal},
		%$$j(\WD(\Max(\pp)))=\ua_{\WD(\pp)}\eta_{\pp}(\Max(\pp)).$$
		%It follows that
		\begin{align*}
			\WD(\pp) & =j(\WD(\Max(\pp)))\cup \eta_{\pp}(\pp\backslash \Max (\pp)) \\
			&=\{\cl_{\pp}(B):B\in \WD ( \Max (\pp))\}\cup\{\da (x,e): (x,e)\in \pp\backslash \Max (\pp)\}.
		\end{align*}
		From the above equation, there exists $B\in \WD ( \Max (\pp))$ such that $A=\cl_{\pp}(B)$.
		Indeed, by Lemma \ref{closure},
		\begin{align*}
			A^*&=\cl_{\pp}(A^*)\cap \Max (\pp)\\
			&=A\cap \Max (\pp)\\
			&=\cl_{\pp}(B)\cap \Max (\pp)\\
			&=B.
		\end{align*}
		So $A^*\in \WD ( \Max (\pp))$. Note that
		$$\ua_{\WD(\pp)}\eta_{\pp}(\Max(\pp))=\{B\in \WD(\pp):B\cap\Max(\pp)\neq\emptyset\}.$$
		By (P4), we have
		\begin{align*}
			\WD(\Max(\pp)) &  =j^{-1}(\ua_{\WD(\pp)}\eta_{\pp}(\Max(\pp)))\\
			&=\{B\cap\Max(\pp):~B\cap\Max(\pp)\neq\emptyset {\rm ~and}~ B\in \WD(\pp)\}.
		\end{align*}
	\end{proof}
	
	\begin{proposition}\label{wd}
		Let $P$ be a bounded complete algebraic poset. Then $\Max(\pp)$ is a WD space if and only if $\Sigma\pp$ is a WD space.
	\end{proposition}
	\begin{proof}
		Suppose that $\Max(\pp)$ is a WD space. Then $\WD(\Max(\pp))=\Irr(\Max(\pp))$.
		Thus
		\begin{align*}
			\WD(\pp)  & =\ua_{\WD(\pp)}\eta_{\pp}(\Max(\pp))\cup \eta_{\pp}(\pp\backslash \Max (\pp)) &({\rm~by ~(P2)})\\
			&=j(\WD(\Max(\pp)))\cup \eta_{\pp}(\pp\backslash \Max (\pp))&({\rm~by ~(P4)})\\
			&=j(\Irr(\Max(\pp)))\cup \eta_{\pp}(\pp\backslash \Max (\pp))&({\rm ~by ~ hypothesis})\\
			& =\ua_{\Irr(\pp)}\eta_{\pp}(\Max(\pp))\cup \eta_{\pp}(\pp\backslash \Max (\pp))&({\rm ~by ~ Proposition~} \ref{prop:embed})\\
			&=\Irr(\pp).&({\rm ~by ~ Corollary~ }\ref{equation0})
		\end{align*}
		
		For the converse, if $\Sigma\pp$ is a WD space, then $\WD(\pp)=\Irr(\pp)$.
		According to Corollary \ref{equation3} and Proposition \ref{prop:embed}, we have
		\begin{align*}
			\WD(\Max(\pp)) & =j^{-1}(\ua_{\WD(\pp)}\eta_{\pp}(\Max(\pp))) \\
			& =j^{-1}(\ua_{\Irr(\pp)}\eta_{\pp}(\Max(\pp)))\\
			&=\Irr(\Max(\pp)).
		\end{align*}
	\end{proof}
		Immediately, we can obtain the main result.
	\begin{theorem}
		A $T_1$ space $X$ is a WD space if and only if  its Xi-Zhao model endowed with the Scott topology  is a WD space.
	\end{theorem}
	\section{A further discussion}
	
	In this section, we introduce a new $T_0$ topological space ---- the ``WK spaces" lying between $T_0$ spaces and ``Rudin spaces", and provide a uniform way to  some properties known to be preserved by the Xi-Zhao model.
	
	\begin{definition}
		A $T_0$ space $X$ is called a \emph{WK space} if every closed WD set is a KF-set, that is, $\KF(X)=\WD(X)$.
	\end{definition}
	
	Applying Lemmas \ref{wf} and \ref{subset}, it is clear that every well-filtered spaces is a  WK space and every Rudin space is a WK space. The next two examples illustrate that these implications are strict.
	\begin{example}
		Consider the set $\R$ of all real numbers equipped with the co-countable topology $\tau_{coc}$.
		The space $(\R, \tau_{coc})$ is well-filtered, then it is a WK space. But it is not a Rudin space.
	\end{example}
	
	\begin{example}\label{cof}
		Consider the set $\n$ of all natural numbers endowed with the cofinite topology $\tau_{cof}$.
		The space $(\n, \tau_{cof})$ is  a Rudin space, then it is a WK space. But  it is not well-filtered.
	\end{example}
	Therefore, WK spaces are not always Rudin spaces or well-filtered spaces.
How about the relationship between WK spaces and WD spaces?
	Naturally, there are two questions as follows.
	\begin{question}
		Is any closed WD set  a KF-set in a $T_1$ space?
	\end{question}
	\begin{question}
		Is any $T_1$  WK  space  a Rudin  space?
	\end{question}
	
%	It has shown that  a $T_1$ space $X$ is a Rudin space (resp., sober space, well-filtered space) if and only if the Scott topology of its Xi-Zhao model  is a Rudin space (resp., sober space, well-filtered space)(see \cite{chen22,zhao18,xi17}).
Like WD spaces,	we will prove that a $T_1$ space $X$ is a WK space if and only if its Xi-Zhao model endowed with the Scott topology  is a  WK space.

	\begin{proposition}\label{wk}
		Let $P$ be a bounded complete algebraic poset. Then $\Max(\pp)$ is a WK space if and only if $\Sigma\pp$ is a WK space.
	\end{proposition}
	\begin{proof}
		Suppose that $\Max(\pp)$ is a WK space. Then $\WD(\Max(\pp))=\KF(\Max(\pp))$.
		It follows that
		\begin{align*}
			\WD(\pp)  & =\ua_{\WD(\pp)}\eta_{\pp}(\Max(\pp))\cup \eta_{\pp}(\pp\backslash \Max (\pp)) \\
			&=j(\WD(\Max(\pp)))\cup \eta_{\pp}(\pp\backslash \Max (\pp))\\
			&=j(\KF(\Max(\pp)))\cup \eta_{\pp}(\pp\backslash \Max (\pp))\\
			& =\ua_{\KF(\pp)}\eta_{\pp}(\Max(\pp))\cup \eta_{\pp}(\pp\backslash \Max (\pp))\\
			&=\KF(\pp).
		\end{align*}
		
		For the converse, if $\Sigma\pp$ is a WD space, then $\WD(\pp)=\Irr(\pp)$.
		According to Corollary \ref{equation3} and Proposition \ref{prop:embed}, we have
		\begin{align*}
			\WD(\Max(\pp)) & =j^{-1}(\ua_{\WD(\pp)}\eta_{\pp}(\Max(\pp))) \\
			& =j^{-1}(\ua_{\Irr(\pp)}\eta_{\pp}(\Max(\pp)))\\
			&=\Irr(\Max(\pp)).
		\end{align*}
	\end{proof}
	Immediately, we can obtain the following conclusions.
	\begin{theorem}
		A $T_1$ space $X$ is a WK space if and only if  its Xi-Zhao model endowed with the Scott topology  is a WK space.
	\end{theorem}

	From the proofs of Propositions \ref{wd} and \ref{wk}, there are certain rules to follow.
	To develop a general framework for proving these properties in the  Xi-Zhao model, we introduce the following  concepts.
	
	Let ${\bf Top_0}$ be the category of all $T_0$ spaces with continuous mappings and ${\bf Set}$ the category of all sets with mappings.
	
	\begin{definition}(\cite{xu_h-sober_2021})
		A covariant functor ${\rm H} : {\bf Top_0} \longrightarrow {\bf Set}$ is called a \emph{subset system} on ${\bf Top_0}$
		provided that the following two conditions are satisfied:
		\begin{itemize}
			\item[(1)] ${\mathcal{S}}(X)\subseteq {\rm H} (X) \subseteq 2^X$ for each $T_0$ space $X$.
			\item[(2)]  For any continuous mapping $f:X\longrightarrow Y$ in ${\bf Top_0}$, ${\rm H} (f)(A)=f(A)\in {\rm H} (Y)$ for all $A\in {\rm H} (X)$.
		\end{itemize}
	\end{definition}
	For a subset system ${\rm H}  : {\bf Top_0} \longrightarrow {\bf Set}$ and a  $T_0$ space $X$, let ${\rm H}_c (X) = \{\overline{A} : A \in {\rm H} (X)\}.$
	\begin{definition}(\cite{xu_h-sober_2021})
		A subset system ${\rm H}: {\bf Top_0} \longrightarrow {\bf Set}$ is called an \emph{irreducible subset system},
		or an \emph{R-subset system} for short, if ${\rm H}(X) \subseteq {\rm Irr}(X)$ for each $T_0$ space $X$.
	\end{definition}
	From Lemma \ref{p123},  the following result holds immediately.
	\begin{lemma}
		For an R-subset system ${\rm H}: {\bf Top_0} \longrightarrow {\bf Set}$ and a bounded complete algebraic poset $P$, if ${\rm H}_c(\pp)\subseteq{\rm H}(\pp)$, then the pair $\langle P_H({\rm H}_c(\pp)), \eta_{\pp}\rangle$  satisfies (P1)$\sim$(P3).
	\end{lemma}
	\begin{definition}
		A subset system ${\rm H}: {\bf Top_0} \longrightarrow {\bf Set}$ is called  \emph{dcpo model determined},
		if ${\rm H}_c(X)\subseteq{\rm H}(X)$ for each $T_0$ space $X$, and the pair $\langle P_H({\rm H}_c(\pp)),\eta_{\pp}\rangle$ satisfies (P1), (P2) and
		\begin{itemize}
			\item[(P4)] $j({\rm H}_c(\Max(\pp)))=\ua_{{\rm H}_c(\pp)}\eta_{\pp}(\Max (\pp))$, where $j(A)=\cl_{\pp}(A)$ for all $ A\in {\rm H}_c(\Max(\pp))$, 
		\end{itemize}
		for each  bounded complete algebraic poset $P$.
	\end{definition}
	For a  dcpo model determined subset system ${\rm H}$,  the map $j:({\rm H}_c(\Max(\pp)),\subseteq)\longrightarrow ({\rm H}_c(\pp),\subseteq)$ is well-defined,  and in addition,  it is an order embedding since (P4) holds. Naturally, each of the subset systems $\mathcal{S}$, $\kf$, $\wds$ and $\irr$ is    dcpo model determined.
	There are some good properties that dcpo model determined subset systems have.
	
	%	\begin{corollary}
		% The pairs $\langle P_H(\si(\pp)), \eta_{\pp}\rangle$, $\langle P_H(\KF(\pp)), \eta_{\pp}\rangle$, $\langle P_H(\WD(\pp)), \eta_{\pp}\rangle$ and $\langle P_H(\Irr(\pp)), \eta_{\pp}\rangle$ all satisfy (P1)$\sim$(P4).
		%\end{corollary}
		%\begin{proof}
		%	From Claim \ref{j1}, Claim \ref{irr2}, Remark \ref{equation0} and the remark of Claim \ref{p123}, the pair $\langle P_H(\Irr(\pp)), \eta_{\pp}\rangle$  satisfies (P1)$\sim$(P4).
		%	It is easy to see that $X_0=\si(\Max(\pp))$, $Y_0=\si(\pp)$, $X_1=\KF(\Max(\pp))$ and $Y_1=\KF(\pp)$.
		%	Then pairs $\langle P_H(\si(\pp)), \eta_{\pp}\rangle$ and $\langle P_H(\KF(\pp)), \eta_{\pp}\rangle$ both satisfy (P1)$\sim$(P4).
		%	By Claim \ref{p123} to Claim \ref{equal}, the pair $\langle P_H(\WD(\pp)), \eta_{\pp}\rangle$ satisfies (P1)$\sim$(P4).
		%\end{proof}

		\begin{corollary}
			Let $~{\rm H}: {\bf Top_0} \longrightarrow {\bf Set}$ be a dcpo model determined subset system and $P$ a bounded complete algebraic poset.
			Then we have
			$$	{\rm H}_c(\pp) = \{\cl_{\pp}(A):A\in {\rm H}_c ( \Max (\pp))\}\cup\{(\da (x,e): (x,e)\in \pp\backslash \Max (\pp)\}.$$
		\end{corollary}
		Unfortunately, if $A\cap\Max (\pp)$ is nonempty,  $A\cap\Max (\pp)\in{\rm H}_c ( \Max (\pp)) $ does not always hold for all $A\in{\rm H}_c ( \pp)$.
		\begin{proposition}\label{key}
			Let $~{\rm H}, {\rm H'}: {\bf Top_0} \longrightarrow {\bf Set}$ be two dcpo model determined subset systems and $P$ a bounded complete algebraic poset.
			Then  ${\rm H}_c(\pp)={\rm H}'_c(\pp)$   if and only if ${\rm H}_c(\Max(\pp))={\rm H}'_c(\Max(\pp))$.
		\end{proposition}
		\begin{proof}
			Suppose that ${\rm H}_c(\Max(\pp))={\rm H}'_c(\Max(\pp))$.
			According to  (P2) and (P4), we have
			\begin{align*}
				{\rm H}_c(\pp)  & =\ua_{{\rm H}_c(\pp)}\eta_{\pp}(\Max(\pp))\cup \eta_{\pp}(\pp\backslash \Max (\pp)) \\
				&=j({\rm H}_c(\Max(\pp)))\cup \eta_{\pp}(\pp\backslash \Max (\pp))\\
				&=j({\rm H}'_c(\Max(\pp)))\cup \eta_{\pp}(\pp\backslash \Max (\pp))\\
				& =\ua_{{\rm H}'_c(\pp)}\eta_{\pp}(\Max(\pp))\cup \eta_{\pp}(\pp\backslash \Max (\pp))\\
				&={\rm H}'_c(\pp).
			\end{align*}
			
			For the converse, assume that ${\rm H}_c(\pp)={\rm H}'_c(\pp)$.
			By (P4), we have
			\begin{align*}
				{\rm H}_c(\Max(\pp)) & =j^{-1}(\ua_{{\rm H}_c(\pp)}\eta_{\pp}(\Max(\pp))) \\
				& =j^{-1}(\ua_{{\rm H}'_c(\pp)}\eta_{\pp}(\Max(\pp)))\\
				&={\rm H}'_c(\Max(\pp)).
			\end{align*}
		\end{proof}
		\begin{definition}
			A $T_0$ space $X$ is called an \emph{${\rm H}$-model} space, if  there exist two different dcpo model determined subset systems $~{\rm H}, {\rm H}': {\bf Top_0} \longrightarrow {\bf Set}$ such that ${\rm H}_c(X)={\rm H}'_c(X)$.
		\end{definition}
		That two subset systems $~{\rm H}$, ${\rm H'}$ are different means that there exists a $T_0$ space $Y$ such that ${\rm H}_c(Y)\neq{\rm H}'_c(Y)$.
		In particularly, the discrete space $X=\{x\}$ is an H-model space.
		
		According to Proposition \ref{key}, we have the  following theorem.
		\begin{theorem}
			A $T_1$ space is an ${\rm H}$-model space  if and only if  its Xi-Zhao model endowed with the Scott topology is an H-model space.
		\end{theorem}
		Clearly, sober spaces, Rudin spaces, WD spaces and WK spaces are ${\rm H}$-model spaces due to their definitions.	
		According to Lemma \ref{wf},  well-filtered spaces are  H-model spaces, and then we can conclude the following result whose proof is different from Theorem 3 in \cite{xi17}.
		\begin{proposition}(\cite{xi17})\label{wf2}
			Let $P$ be a bounded complete algebraic poset. Then $\Max(\pp)$ is  well-filtered  if and only if $\Sigma\pp$ is  well-filtered.
		\end{proposition}
		Figure \ref{picture} reveals certain relations among some kinds of H-model spaces, which contain all combinations of dcpo model determined subset systems we know in the previous paragraphs.
		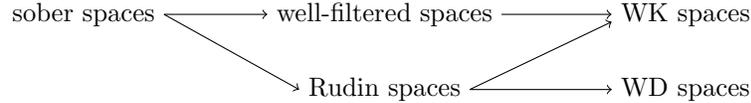
\begin{figure}[H]
			\centering
			\tikzstyle{format}=[rectangle,draw=white,thin,fill=white]		
			\tikzstyle{test}=[diamond,aspect=2,draw,thin]
			\tikzstyle{point}=[coordinate,on grid,]
			\begin{tikzpicture}
				\node[format] (sober){sober spaces};
				\node[format,right of=sober,node distance=40mm] (wf){well-filtered spaces};
				\node[format,right of=wf,node distance=40mm] (xu){WK spaces};
				\node[format,below of=wf,node distance=10mm] (rd){Rudin spaces};				
				\node[format,below of=xu,node distance=10mm] (wds){WD spaces};				
				\draw[->] (sober.east)--(wf.west);
				\draw[->] (sober.east)--(rd.west);
				\draw[->] (rd.east)--(wds.west);
				\draw[->] (rd.east)--([yshift=-1mm]xu.west);
				\draw[->] (wf.east)--(xu.west);
			\end{tikzpicture}
			\caption{Certain relations among some examples of H-model spaces.}
			\label{picture}
		\end{figure}
		Besides that, are there other dcpo model determined subset systems and other kinds of $T_0$ spaces the Xi-Zhao model preserves?		
		Inspired by the work on weak sober spaces and weak well-filtered spaces in \cite{chen22,chen_xi-zhao_2023}, we define a new kind of spaces that are weak than  H-model spaces and can be preserved by the Xi-Zhao model.
		
		%		Let $X$ be a $T_0$ space. We denote $\irr^*(X)$ (resp., $\kf^*(X)$, $\wds^*(X)$) the set of all proper irreducible subsets (resp., proper KF-subsets, proper WD subsets) of $X$ if $X$ is not singleton. If $X=\{x\}$, set $\irr^*(X)=\kf^*(X)=\wds^*(X)=\{X\}$.
		For a subset system ${\rm H}$ and a  $T_0$ space $X$, let ${\rm H}^*(X) ={\rm H}(X)\backslash \{X\}$ and ${\rm H}^*_c(X) ={\rm H}_c(X)\backslash \{X\}.$ 		
		The following lemma holds by Lemma \ref{subset}.
		\begin{lemma}\label{subset2}
			Let $X$ be a $T_0$ space. Then $\si^*(X)\subseteq \KF^*(X)\subseteq \WD^*(X)\subseteq\Irr^*(X)$.
		\end{lemma}
		Anyway,  $\mathcal{S}^*$, $\irr^*$, $\kf^*$ and $\wds^*$ are not subset systems.

		%A  subset system ${\rm H}: {\bf Top_0} \longrightarrow {\bf Set}$ is said to satisfy  condition ($\ast$) if :
		%\begin{itemize}
		%	\item[($\ast$)] $\pp\in{\rm H}_c(\pp)$  if and only if  $\Max (\pp)\in{\rm H}_c(\Max (\pp))$ for each  bounded complete algebraic poset $P$.
		%\end{itemize}
		\begin{definition}
			A $T_0$ space $X$ is called a \emph{weak ${\rm H}$-model} space, if  there exist two different dcpo model determined subset systems $~{\rm H}, {\rm G}: {\bf Top_0} \longrightarrow {\bf Set}$ such that ${\rm H}^*_c(X)={\rm G}^*_c(X)$.
		\end{definition}
		It is obvious that every  H-model space is a weak  H-model space.
		Similarly, we can get the following results.
		\begin{corollary}\label{key2}
			Let $~{\rm H}, {\rm G}: {\bf Top_0} \longrightarrow {\bf Set}$ be two dcpo model determined subset systems and $P$ a bounded complete algebraic poset.
			Then  ${\rm H}^*_c(\pp)={\rm G}^*_c(\pp)$   if and only if $~{\rm H}^*_c(\Max(\pp))={\rm G}^*_c(\Max(\pp))$.
		\end{corollary}
		\begin{corollary}
			A $T_1$ space is a weak ${\rm H}$-model space  if and only if  its Xi-Zhao model endowed with the Scott topology is a weak H-model space.
		\end{corollary}
		%\begin{lemma}(\cite{chen22})\label{weakirr}
		%	Let $P$ be a poset. Then $P\in \Irr(P)$ if and only if  $\Max (P)\in\Irr(\Max (P))$.
		%\end{lemma}
		%By Corollary \ref{kfset} and Corollary \ref{equation3}, we have the following result.
		%\begin{lemma}\label{weak}
		%	Let $P$ be a bounded complete algebraic poset. Then
		%	\begin{itemize}
			%		\item[(1)] $\pp\in\KF(\pp)$  if and only if  $\Max (\pp)\in\KF(\Max (\pp))$;
			%		\item[(2)] $\pp\in\WD(\pp)$ if and only if  $\Max (\pp)\in\WD(\Max (\pp))$.
			%	\end{itemize}
		%\end{lemma}
		
		Two kinds of weak ${\rm H}$-model spaces are shown as follows.
		\begin{definition}(\cite{lu_equality_2019})
			A topological space $X$ is called \emph{weak sober} if for any  proper irreducible closed subset $A$, there exists a unique  point $x \in X$ such that $A=\overline{\{x\}}$, that is,  $\Irr^*(X)=\si^*(X)$.
		\end{definition}
		Then weak sober spaces are weak H-model spaces.
		
		\begin{definition}(\cite{lu17})
			A topological space $X$ is called \emph{weak well-filtered} if for any filtered  intersection $\bigcap_{i \in I} Q_i$ of compact saturated subsets and its any nonempty open neighborhood $U$, $U$ contains $Q_i$ for some $i$.
		\end{definition}
		\begin{lemma}
			For a  $T_0$ space $X$, the following conditions are equivalent:
			\begin{enumerate}[(1)]
				\item $X$ is weak well-filtered;	
				\item $\KF^*(X)=\si^*(X)$.
			\end{enumerate}
		\end{lemma}
		\begin{proof}
			If $X$ is a singleton set, then $\KF^*(X)$ and $\si^*(X)$ are empty. We only need to consider the nontrivial case.
			
			(1)$\Rightarrow$(2) For $A\in \KF^*(X)$, there exists a filtered family $\K\subseteq Q(X)$ such that  $A\in m(\K)$.
			Suppose that $A\cap(\bigcap\K)$ is empty. Then $\bigcap\K\subseteq X\backslash A\neq \emptyset$.
			It follows that $X\backslash A $ contains $K_0$ for some  $K_0\in\K$, which contradicts  the condition that $A\cap K_0$ is nonempty.
			Pick $a\in A\cap(\bigcap\K)$. Then $\da a$ is a closed subset of $A$ that intersects all members of $\K$.
			We have $A=\da a\in \si^*(X)$ by the minimality of $A$.
			
			(2)$\Rightarrow$(1) Let $\K$ be a filtered family of compact saturated subsets and $U$ a nonempty open set with $\bigcap\K\subseteq U$.
			Suppose that $K\nsubseteq U$ for all $K\in\K$. Then there exists a closed subset $A$ of $X\backslash U$ such that $A\in m(\K)$ by Lemma \ref{rudinlemma}. Thus $A\in\KF^*(X)$  since $U$ is nonempty. It implies that $A=\da a$ for some $a\in A$.
			Hence, $a\in K$  because $A\cap K\neq \emptyset$ and $K$ is an upper set for all $K\in \K$.
			We have $a\in \bigcap\K\subseteq U$, which contradicts to $a\in A\subseteq X\backslash U$.
			Thus there is some $K_0\in \K$ such that $K_0\subseteq U$.
			So $X$ is weak well-filtered.
		\end{proof}
		Obviously, every weak well-filtered space is a weak H-model space.
		
		Similarly, by combining two of these operators mentioned in Lemma \ref{subset2}, we can obtain other examples of weak H-model spaces, and all of them can be preserved by the Xi-Zhao model.
		To avoid repetition, the detail descriptions are not displayed in this paper.
		
		%\section*{Acknowledgements}
		%We would like to thank the anonymous reviewers for their helpful comments and valuable suggestions.
		
		%\section*{Declarations}
		\section*{References}
		\bibliographystyle{unsrt}%{plain}%

	\end{document}